\documentclass[12pt]{article}
\usepackage[T1]{fontenc}
\usepackage{amssymb, amsmath, amsthm}
\usepackage[all]{xy}
\usepackage{newtxtext,stmaryrd}
\usepackage{hyperref}
\usepackage{accents}
\usepackage{tikz-cd}
\usepackage{hhline}
\usepackage{multicol}
\usepackage{mathtools}

\usetikzlibrary{backgrounds,matrix}

\newcommand{\ttmat}[4]{\begin{pmatrix}
#1 & #2 \\
#3 & #4
\end{pmatrix}}

\newcommand{\sm}[4]{\ensuremath{\left(\begin{smallmatrix}#1 & #2 \\ #3 & #4\end{smallmatrix}\right)}}

\newcommand{\smthree}[3]{\ensuremath{\left(\begin{smallmatrix}1 & #1 & #2  \\ 0& 1 & #3 \\
0 & 0 & 1\end{smallmatrix}\right)}}

\DeclareFontEncoding{OT2}{}{}

\newcommand{\mr}[1]{\mathrm{#1}}
\newcommand{\mf}[1]{\mathfrak{#1}}
\newcommand{\mc}[1]{\mathcal{#1}}
\newcommand{\mb}[1]{\mathbb{#1}}

\newcommand{\Z}{\mb{Z}}
\newcommand{\Q}{\mb{Q}}
\newcommand{\zp}{\mb{Z}_p}

\newcommand{\fp}{\mb{F}_p}

\newcommand{\F}{\mb{F}}

\newcommand{\ab}{\mr{ab}}

\newcommand{\sdot}{\,\cdot\,}

\newcommand{\ps}[1]{\llbracket #1 \rrbracket}
\newcommand{\Iw}{\mr{Iw}}
\newcommand{\GL}{\mr{GL}}

\newcommand{\fblock}[1]{\begin{array}{|@{\:}cc}
#1 \\
\hhline{|--|}
\end{array}}

\newcommand{\binmat}[2]{\genfrac{[}{]}{0pt}{}{#1}{#2}}

\DeclareMathOperator{\Hom}{Hom} \DeclareMathOperator{\Aut}{Aut}
\DeclareMathOperator{\End}{End} \DeclareMathOperator{\Gal}{Gal}
 
\DeclareMathOperator{\coker}{coker}  
\DeclareMathOperator{\image}{im}

\DeclareMathOperator{\Tot}{Tot}

\newcommand{\up}[1]{^{(#1)}}
\newcommand{\Cl}{\mathrm{Cl}}

\usepackage{color}

\newtheorem{theorem}{Theorem}[subsection]
\newtheorem{proposition}[theorem]{Proposition}
\newtheorem{lemma}[theorem]{Lemma}
\newtheorem{corollary}[theorem]{Corollary}
\newtheorem*{thmA}{Theorem A}
\newtheorem*{thmB}{Theorem B}
\newtheorem*{thmC}{Theorem C}

\theoremstyle{definition}
\newtheorem{definition}[theorem]{Definition}

\theoremstyle{remark}
\newtheorem{remark}[theorem]{Remark}
\newtheorem*{ack}{Acknowledgments}
\newtheorem*{fin}{Financial Support}
\newtheorem{example}[theorem]{Example}

\numberwithin{equation}{section}

\renewcommand{\baselinestretch}{1.2}

\begin{document}

\title{Generalized Bockstein maps and Massey products}
\author{ \ \ 
Yeuk Hay Joshua Lam\thanks{Institut des Hautes \'Etudes Scientifiques,
	35 Route de Chartres,
	Bures-sur-Yvette 91440,
	France,
	\texttt{ylam@ihes.fr}}
\and 
Yuan Liu\thanks{Department of Mathematics, University of Illinois Urbana-Champaign, 1409 W. Green Street, Urbana, IL\ \ 61801, \texttt{yyyliu@illinois.edu}} 
\and 
Romyar Sharifi\thanks{Department of Mathematics, University of California, Los Angeles, 520 Portola Plaza, Los Angeles, CA\ \ 90095, \texttt{sharifi@math.ucla.edu}} \ \
\and
Preston Wake\thanks{Department of Mathematics, Michigan State University, 619 Red Cedar Road, East Lansing, MI\ 48824, \texttt{wakepres@msu.edu}} 
\and
Jiuya Wang\thanks{
	Department of Mathematics,
	University of Georgia,
	Boyd Graduate Studies Research Center,
	Athens, GA\ \ 30602,
	\texttt{jiuya.wang@uga.edu}}
}
\date{}
\maketitle

\begin{abstract}
Given a profinite group $G$ of finite $p$-cohomological dimension and a pro-$p$ quotient $H$ of $G$ by a closed normal subgroup $N$, we study the filtration on the Iwasawa cohomology of $N$ by powers of the augmentation ideal in the group algebra of $H$. We show that the graded pieces are related to the cohomology of $G$ via analogues of Bockstein maps for the powers of the augmentation ideal.  For certain groups $H$, we relate the values of these generalized Bockstein maps to Massey products relative to a restricted class of defining systems depending on $H$.  We apply our study to prove lower bounds on the $p$-ranks of class groups of certain nonabelian extensions of $\Q$ and to give a new proof of the vanishing of Massey triple products in Galois cohomology.
\end{abstract}
\renewcommand{\thefootnote}{\arabic{footnote}} 
\renewcommand{\thefootnote}{\fnsymbol{footnote}} 
\footnotetext{\!\!\textbf{MSC2020:} 20J05, 20J06, 12G05 (Primary), 11R23, 11R34 (Secondary)} 

\newpage
\tableofcontents

\section{Introduction}

At its essence, this paper revolves around the fundamental question: 
\begin{quote} \emph{How does the continuous cohomology of a profinite group $G$ with compact coefficients compares with the cohomology of an open normal subgroup $N$?}
\end{quote}

As a starting point, if $G$ has cohomological dimension $d$, then 
corestriction induces an isomorphism from the $G/N$-coinvariants of a $d$th cohomology group of $N$ to the $d$th cohomology group of $G$ with the same coefficients. We view this corestriction map as the first of a sequence of generalized Bockstein maps $\Psi^{(n)}$ for $n \ge 0$, which we extend to closed $N$ by considering Iwasawa cohomology. The powers of the augmentation ideal $I$ of a completed group ring of $G/N$ yield a natural filtration on the domain of corestriction. In Section \ref{sec:bock}, we show that the $n$th graded piece of this augmentation filtration is isomorphic to the cokernel of $\Psi^{(n)}$, employing two purely homological results of Appendix \ref{appendix} in the proof.  In Section \ref{massey}, we demonstrate how, in many cases, the image of $\Psi^{(n)}$ is described by $n$-fold Massey products. 

Massey products were first introduced by Massey \cite{massey58} as a tool for proving that two topological spaces are not homotopy equivalent even when they have isomorphic cohomology rings. 
The best-known example involves the complement of the Borromean rings in $\mathbb{R}^3$, three pairwise unlinked circle which are nonetheless linked, resulting in a nontrivial Massey triple product in the second cohomology.
In algebra, Massey products are used to study properties of a group $G$ that are not detected by the group cohomology ring itself. Massey products of tuples of homomorphisms on $G$ valued in a ring $R$ are obstructions in $H^2(G,R)$ to lifting homomorphisms to unipotent matrices from the quotient by the center. 

Our initial motivation for studying this question came from Iwasawa theory. 
Indeed, Galois groups of number fields with restricted ramification above a prime $p$ have $p$-cohomological dimension equal to $2$, and their second cohomology groups with coefficients in $p$-power roots of unity are closely related to ideal class groups. In such a setting, the fundamental question above translates to comparing ideal class groups as one goes up a tower of fields, the original question of Iwasawa theory. In this vein, Mazur \cite{mazur} described an analogy between knot complements in $3$-manifolds and Galois groups with restricted ramification, relating the Alexander polynomial of a knot and a characteristic ideal of an inverse limit of class groups. Morishita explored this analogy in terms of Massey products (see for example \cite{morishita2004}).

The third author studied Massey products in an Iwasawa-theoretic context, relating them to the structure of augmentation-graded pieces of limits of class groups in a nonabelian tower of Kummer extensions \cite{sh-massey}. This paper distills the purely algebraic results of the latter paper from their number-theoretic application. The distinct perspective using generalized Bockstein maps, that we introduce here, allows us to go beyond the procyclic setting of \cite{sh-massey}. 

Massey products of length $n$ are defined only if certain $(n-1)$-fold Massey products vanish. Even when defined, there is some indeterminacy in their definition, resulting from a choice of defining system, a homomorphism to the quotient of the $(n+1)$-dimensional unipotent matrices by their center. The Massey product provides the obstruction to lifting this homomorphism to the full unipotent group. In order to view $n$-fold Massey products as values of $\Psi^{(n)}$, we define an appropriate notion of a proper defining system, reducing the aforementioned indeterminacy. The requisite definitions are given in some generality in Section \ref{sec:masseyreview}, providing the framework for the comparison with $\Psi^{(n)}$ in specific cases described in Section \ref{massey}.

In Section \ref{sec:cyclo}, we demonstrate how our methods can be used to derive concrete arithmetic results by proving lower bounds on the $p$-ranks of class groups of finite $p$-ramified bicyclic and Heisenberg extensions of $\Q(\mu_p)$. Though we eschew Iwasawa-theoretic applications in this paper to ground our study, a description of the augmentation filtrations of inverse limits of $p$-parts of class groups in $\zp$-extensions, 
derived using our methods, may be found in \cite{reciprocity}. 

We also consider applications of generalized Bockstein maps to the study of absolute Galois groups of fields. Many algebraic properties of absolute Galois groups are encoded cohomologically as properties of the norm residue symbol. The celebrated norm residue isomorphism theorem of Voevodsky and Rost \cite{voevodsky} (formerly the Milnor-Bloch-Kato conjecture), describes cohomology rings of absolute Galois groups with coefficients in twists of roots of unity as Milnor $K$-rings of the fields. 

The Massey vanishing conjecture of  Min\'{a}\v{c} and T\^{a}n \cite{mt1} goes beyond the cohomological ring structure to posit that, for $n \ge 3$, all definable $n$-fold Massey products with $\F_p$-coefficients vanish for some choice of defining system. Earlier work of Hopkins--Wickelgren \cite{hw} had established this for $n = 3$ and $p = 2$ over number fields. The full $n = 3$ case of this conjecture is the triple Massey vanishing theorem of Efrat--Matzri \cite{em} and Min\'{a}\v{c}--T\^{a}n \cite{mt2}. 
The introduction to Section \ref{masseyvanish} provides a more detailed, yet still incomplete, summary of the history of and rapid progress in this area.
In that section, we show that certain algebraic properties of absolute Galois groups are naturally expressed in terms of generalized Bockstein maps. This perspective enables us to
give a new proof for odd primes of the triple Massey vanishing theorem.

Min\'{a}\v{c} and T\^{a}n originally formulated the Massey vanishing conjecture, in part, as a way to help cohomologically characterize which profinite groups are isomorphic to absolute Galois groups of fields.  We suspect that generalized Bockstein maps have an important role to play in formulating and understanding such cohomological characterizations.

We next provide a more detailed overview of our main results.

\subsection{Comparing cohomology using generalized Bockstein maps}
\label{introBock}
Let $G$ be a profinite group of $p$-cohomological dimension $d \ge 1$.  Let $H$ be a finitely generated pro-$p$ quotient of $G$ by a closed normal subgroup $N$.  Let $\Omega = \zp\ps{H}$ denote the completed $\zp$-group ring of $H$, the inverse limit of the $\zp$-group rings of the finite quotients of $H$.  Let $T$ be a finitely generated $\zp$-module with a continuous action of $G$.  This paper is concerned  with the study of connecting maps in the continuous $G$-cohomology of the augmentation filtration of the tensor product $T \otimes_{\zp} \Omega$.  That is, if $I = \ker(\Omega \to \zp)$ denotes the augmentation ideal of $\Omega$, then we have exact sequences
\begin{equation} \label{keyexseq}
	0 \to T \otimes_{\zp} I^n/I^{n+1} \to T \otimes_{\zp} \Omega/I^{n+1} \to T \otimes_{\zp} \Omega/I^n \to 0
\end{equation}
for each $n \ge 1$ such that $\Omega/I^n$ is $\zp$-flat. Our interest lies in the connecting homomorphisms
$$
	\Psi^{(n)} \colon H^{d-1}(G,T \otimes_{\Z_p} \Omega/I^n) \to H^d(G,T) \otimes_{\Z_p} I^n/I^{n+1}
$$
attached to these sequences, which we refer to as \emph{generalized Bockstein maps}, due to their similarlity to usual Bockstein maps
for exact sequences of $p$-power order cyclic groups.

We can use the Bockstein maps to partially describe the second \emph{Iwasawa cohomology group} $H^d_{\Iw}(N,T)$ of $N$ with $T$-coefficients.  This cohomology group is the inverse limit of the groups $H^d(U,T)$ under corestriction maps, where $U$
runs over the open normal subgroups of $G$ containing $N$.   It is naturally endowed, through the $\zp[G/U]$-actions on each $H^d(U,T)$, with the structure of an $\Omega$-module.  We prove that the cokernels of the generalized
Bockstein maps describe the graded quotients in the augmentation filtration of $H^d_{\Iw}(N,T)$ (see Theorem \ref{augfilt}).

\begin{thmA}
	There are canonical isomorphisms
	$$
		\label{eq:intro main thm}
		\frac{I^nH^d_\Iw(N,T) }{I^{n+1}H^d_\Iw(N,T)} \cong \frac{H^d(G,T) \otimes_{\Z_p} I^n/I^{n+1}}{\image\Psi^{(n)}}.
	$$
\end{thmA}

The proof rests on an Iwasawa-cohomological version \cite{ls, fk} of a descent spectral sequence of Tate, applied to the terms of our exact sequences for the augmentation filtration of $\Omega$.  We verify the compatibility of these spectral sequences with generalized Bockstein maps and a connecting map in the $H$-homology of the $\zp$-tensor product of $H^d_\Iw(N,T)$ with \eqref{keyexseq}. 

\subsection{A brief primer on Massey products}

\label{introMassey}

Given a commutative ring $R$, a \emph{Massey product} $(\chi_1, \ldots, \chi_n)$ of $n$ homomorphisms $\chi_1, \ldots, \chi_n$ in $H^1(G,R)$ is an element of $H^2(G,R)$ that
provides the obstruction to a certain problem of lifting a homomorphism formed using the tuple of characters $\chi_i$ to a homomorphism $\rho \colon G \to \mr{U}_{n+1}(R)$ of $G$ to the group of $(n+1)$-dimensional unipotent matrices in $R$, with $\chi_i$ providing the $i$th off-diagonal entry $\rho_{i,i+1}$.
 
More precisely a \emph{defining
system} for a Massey product $(\chi_1, \ldots, \chi_n)$ is a homomorphism $\rho \colon G \to \mr{U}'_{n+1}(R)$ 
to the quotient of $\mr{U}_{n+1}(R)$ by its center, with $\rho_{i,i+1} = \chi_i$.
The Massey product 
$(\chi_1, \ldots, \chi_n)_{\rho}$ of $\chi_1, \ldots,
\chi_n$ relative to the defining system $\rho$ is the class in $H^2(G,R)$ of the $2$-cocycle 
$$
	F \colon (\sigma,\tau) \mapsto \sum_{i=1}^n \rho_{1,i}(\sigma)\rho_{i,n+1}(\tau).
$$
It vanishes if and only if $\rho$ lifts to a homomorphism $\tilde{\rho} \colon G \to \mr{U}_{n+1}(R)$. In other words, 
the Massey product relative to $\rho$ is the obstruction to choosing the remaining upper right-hand entry $\tilde{\rho}_{1,n+1}$ 
to make $\tilde{\rho}$ a homomorphism, which is exactly to say that $d\tilde{\rho}_{1,n+1} = -F$.

An $n$-fold Massey product $(\chi_1, \ldots, \chi_n)$ is said to be \emph{defined} if a defining system for it exists.
For $n=2$, the Massey product is defined and equals the cup product $\chi_1 \cup \chi_2$. For $n \ge 3$, 
a Massey product need not be defined, and even if it is, 
it may have \emph{indeterminacy} in its values, coming from the different choices of defining systems.
A Massey product is said to \emph{contain zero} or \emph{vanish} if it has a defining system for which the Massey product is zero.

We shall work with profinite groups and compact coefficient rings, so our Massey products take values in continuous cohomology groups, and all cocycles and homomorphisms involved are required to be continuous. In fact, we shall allow more general Massey products valued in modules over a group ring, replacing the group of unipotent matrices with an analogous object in a generalized matrix algebra.

\subsection{The images of generalized Bockstein maps}

The case $d = 2$ and $H \cong \zp$ of Theorem A was first studied in \cite{sh-massey} from a different perspective and applied in an Iwasawa-theoretic context.  Its main result has a similar form to Theorem A, but in place of the image of $\Psi^{(n)}$, it has a group of values of certain $(n+1)$-fold Massey products.
We relate the image of $\Psi^{(n)}$ to Massey products for a variety of groups $H$.  

In the situation of Section \ref{introBock} with $H \cong \Z_p$, the quotient map $G \to H$ can be thought of as an element $\chi \in H^1(G,\Z_p)$. This context was considered in \cite{sh-massey}, and a result like Theorem A is proven, but with the image of $\Psi^{(n)}$ replaced by $(n+1)$-fold Massey products of the form $(\chi,\chi,\dots,\chi, \sdot)$ with respect to certain ``proper'' defining systems. In Section \ref{procyclic}, we show that, in the case that $H$ is procyclic, the image of $\Psi^{(n)}$ is generated by these same Massey products. In other words, when $H$ is procyclic, Theorem~A recovers the main result of \cite{sh-massey}.

This raises the question of whether the relation between the values of generalized Bockstein maps and Massey products can be extended from procyclic $H$ to more general groups. The most difficult step  is to determine the appropriate notion of proper defining system. The key insight is that the proper defining systems of \cite{sh-massey} those defining systems that, in a sense,  partially have group-theoretic origin. That is, if $H$ is procyclic, then for every $n>0$, there is a group homomorphism we call the \emph{unipotent binomial matrix homomorphism}
\[
\binmat{\sdot}{n} \colon H \to \mr{U}_{n+1}(\Z_p)
\]
defined by sending a generator of $H$ to the unipotent matrix with all $1$'s on the diagonal and off-diagonal and $0$'s elsewhere (the notation is meant to evoke binomial coefficients, the non-zero entries of $\binmat{x}{n}$ being binomial coefficients:~see Section \ref{subsec:unipotent binomials}).  A defining system $\rho \colon G \to \mr{U}'_{n+2}(\Z_p)$ for the $(n+1)$-fold Massey product $(\chi,\chi,\dots,\chi, \cdot)$ is called \emph{proper} if its restriction to the upper-left copy of $\mr{U}_{n+1}(\Z_p)$ in $\mr{U}'_{n+2}(\Z_p)$ equals $\binmat{\sdot}{n} \circ \chi$.

This suggests considering defining systems that are, at least partially, of group-theoretic origin.
Let $n \ge 0$ and let $a,b\ge 0$ be such that $a+b=n$. Let
\[
\phi \colon H \to \mr{U}_{a+1}(\Z_p),  \quad \theta \colon H \to \mr{U}_{b+1}(\Z_p)
\]
be group homomorphisms.  By precomposition with $G \to H$, these define an $n$-tuple of elements of $H^1(G,\Z_p)$.  We call that pair $(\phi, \theta)$ a \emph{partial defining system} for $(n+1)$-fold Massey products involving this $n$-tuple of characters.  Our main general result, Theorem \ref{connecting}, is that a partial defining system together with a cocycle $f \in Z^1(G,T \otimes_{\Z_p} \Omega/I^n)$ constitutes a defining system.  Moreover, a partial defining system defines a homomorphism of $G$-modules
\[
p_{\phi,\theta} \colon T \otimes_{\Z_p} I^n/I^{n+1} \to T 
\]
such that $p_{\phi,\theta}(\Psi^{(n)}([f]))\in H^2(G,T)$ is the $(n+1)$-fold Massey product associated to the defining system given by $(\phi,\theta)$ and $f$ (see Theorem \ref{prop:general}).

We apply this general machinery to the procyclic case $H \cong \Z_p$ 
in Section \ref{procyclic}, taking $(a,b) = (n,0)$ and $\phi = \binmat{\sdot}{n}$.  Because $H$ is procyclic, there is an isomorphism $I^n/I^{n+1} \cong \Z_p$ for all $n$, and the map $p_{\binmat{\sdot}{n},1}$ is induced by this isomorphism.  Hence the values $p_{\binmat{\sdot}{n},1}(\Psi^{(n)}([f]))$ completely determine the image of $\Psi^{(n)}$, and in this way we show that the image of $\Psi^{(n)}$ is given by Massey products.

For more general $H$,  the graded quotients $I^n/I^{n+1}$ are more complicated, and we cannot hope for any $p_{\phi,\theta}$ to be an isomorphism.  However, it can happen that $I^n/I^{n+1}$ is a free module; suppose this is the case. If we can arrange that the maps $p_{\phi,\theta}$ for varying $(\phi,\theta)$ give a dual basis to $I^n/I^{n+1}$, then again this construction gives a way to describe the image of $\Psi^{(n)}$ in terms of Massey products. Said differently, if $I^n/I^{n+1}$ is a free module of rank $d$, then, by fixing a basis, we can think of $\Psi^{(n)}([f])$ as a $d$-tuple of elements of $H^2(G,T)$. If we can make $d$ choices of pairs $(\phi,\theta)$ such that the maps $p_{\phi,\theta}$ are the projectors onto these coordinates, then our results describe $\Psi^{(n)}([f])$ as a $d$-tuple of Massey products.

For a general group $H$ and  general $n$, we do not expect that there will exist choices of $(\phi,\theta)$ such that the $p_{\phi,\theta}$ constitute a dual basis to $I^n/I^{n+1}$. However, we give some families of examples where this is the case: $H$ is procyclic (Section \ref{procyclic}), $H$ is pro-bicyclic (Section \ref{bicyc_case}), $H$ is elementary abelian (Section \ref{elem}) and $H$ is a Heisenberg group and $n<4$ (Section \ref{heis_case}).  We explicate the result for $H \cong \Z_p^2$ in the following subsection.

\subsection{The bicyclic case: an illustration}
\label{bicyc image}

Suppose that $H$ is isomorphic to $\zp^2$, and let $\chi, \psi \colon H \to \zp$ denote the projections onto the two factors.  For each nonnegative integer $a \le n$, there is a partial defining system $\left(\binmat{\sdot}{a} \circ \chi, \binmat{\sdot}{n-a}\circ \psi\right)$. 
Applying our general result Theorem \ref{prop:general} with these defining systems, we obtain the following (see Theorem \ref{bicyc_image}):
\begin{thmB}\label{thm:B}
Suppose that $\nu = (\chi,\psi) \colon H \to \zp^2$ is an isomorphism.  Let $x, y \in I$ be such that $x+1$ and $y+1$ are group
elements mapping under $\nu$ to the standard ordered basis of $\zp^2$.  For $n \ge 2$, the cosets of $x^ay^{n-a}$ with $0 \le a \le n$ then form a $\zp$-basis for $I^n/I^{n+1}$.
\begin{enumerate}
	\item[a.] To a continuous $1$-cocycle $f \colon G \to T\otimes_{\zp} \Omega/I^n$ and $0 \le a \le n$,
	we can associate a proper defining system for an $(n+1)$-fold Massey product 
	$$
		(\chi^{(a)},\lambda,\psi^{(n-a)}) := 
		(\underbrace{\chi, \ldots, \chi}_{a\text{ times}}, \lambda, \underbrace{\psi, \ldots, \psi}_{n-a\text{ times}}) \in H^2(G,T),
	$$
	where $\lambda \colon G \to T$
	is the composition of $f$ with the quotient map $T \otimes_{\zp} \Omega/I^n \to T \otimes_{\zp} \Omega/I \cong T$.
	\item[b.] With the notation of part a, let $[f]$ denote the class of $f$ in $H^1(G,T \otimes_{\zp} \Omega/I^n)$.  Then
	$$
		\Psi^{(n)}([f]) = \sum_{a=0}^n (\chi^{(a)},\lambda,\psi^{(n-a)}) \otimes x^ay^{n-a}.
	$$
\end{enumerate}
\end{thmB}

Let us illustrate Theorem B in some detail in the case that $n=2$ and $a=1$. In this case, we have
$$
	\Omega/I^2 = \mathbb{Z}_p[x,y]/(x^2, xy, y^2)
$$ 
in the notation of the theorem.   We can therefore write the $1$-cocycle $f \colon G \to T\otimes_{\mathbb{Z}_p} \Omega/I^2$ as
\[
	f=\lambda+\lambda_x x+\lambda_y y,
\]
with $\lambda_x, \lambda_y \colon G \to T$, abbreviating the tensor product as formal multiplication.
Part a of Theorem B says that $f$ gives rise to a defining system 
$$ 
	\rho = \begin{pmatrix} \begin{matrix} 1 & \chi \\  & 1 \end{matrix} & \fblock{ \lambda_x & * \\ \lambda & \lambda_y } \\ & 
	\begin{matrix} 1& \psi \\  & 1 \end{matrix} 
	\end{pmatrix} \colon G \to U/Z
$$
for the Massey triple product $(\chi, \lambda, \psi)$.  Here, the values of $\rho$ lie in the quotient of a group $U$ of \emph{generalized} upper-triangular unipotent $4$-by-$4$ matrices by its subgroup $Z$ of matrices with zero above-diagonal entries outside of the upper right-hand corner.  The entries in the upper-right hand block are $T$-valued (and in particular $Z \cong T$), whereas they are $\Z_p$-valued outside of it.  Matrix multiplication proceeds using the $\zp$-module structure on $T$. That $\rho$ is a defining system means that $\rho \colon G \to U/Z$ is a nonabelian $1$-cocycle, where $G$ acts on $U$ coordinate-wise.  The Massey product $(\chi,\lambda,\psi)_{\rho}$ relative to the defining system $\rho$ is an element of $H^2(G,T)$ providing the obstruction to lifting $\rho$ to a nonabelian
$1$-cocycle $G \to U$.

In general, even for such a cocycle $\rho$ and therefore a Massey product $(\chi,\lambda,\psi)$ to exist, the cup products $\chi \cup \lambda$ and $\lambda \cup \psi$ must vanish in $H^2(G,T)$ so that cochains $\lambda_x$ and $\lambda_y$ can be chosen with $d\lambda_x = -\chi \cup \lambda$ and $d\lambda_y = -\lambda \cup \psi$.  Even then, the class $(\chi,\lambda,\psi)$ depends on these choices.  In our description, this vanishing is encapsulated in the fact that $f$ is a $1$-cocycle, and the indeterminacy is removed by fixing $f$.

The content of part b of Theorem B is that the coefficients of 
$\Psi^{(2)}([f])$ in $H^2(G,T)$ for the $\zp$-basis $x^2$, $xy$, and $y^2$ of $I^2/I^3$ are Massey triple products: in particular, the coefficient of $xy$ is the Massey product $(\chi,\lambda,\psi)_{\rho}$ 
for the defining system $\rho$.  More precisely, $(\chi,\lambda,\psi)_{\rho}$ is defined as the class of the $2$-cocycle $F \colon G^2 \to T$ given by
$$
	F \colon (g,h) \mapsto \chi(g) g \lambda_x(h) + \psi(h)\lambda_y(g).
$$
This cocycle $F$ arises as the upper-right hand corner of $(g,h) \mapsto \tilde{\rho}(g) \cdot g \tilde{\rho}(h)$ for 
the naive lift of $\rho$ to a cochain $\tilde{\rho} \colon G \to U$ with zero upper-right hand corner.  The theorem boils down to the fact that $F \cdot xy$ is also exactly the coboundary of the naive lift of $f$ to a cochain $G \to \zp[x,y]/(x^2,y^2)$ with 
a zero $xy$-coefficient.

From our perspective, the generalized Bockstein maps are more flexible than Massey products, being connecting homomorphisms more directly amenable to basic applications of homological algebra.  
For instance, the argument proving Theorem A for arbitrary $H$ amounts to a diagram chase for maps of Grothendieck spectral sequences.  Moreover, Theorem B allows us to study defining systems using abelian, rather than nonabelian, cocycles.

\subsection{Galois groups with restricted ramification and class groups}
\label{introclassgroups}
At its core, our work is motivated by the potential arithmetic applications. One has at least something of an understanding of class groups of cyclotomic fields through Bernoulli numbers and thereby $L$-functions, and most notably via the Iwasawa main conjecture (theorem of Mazur-Wiles).  However, little is known about $p$-adic analytic invariants describing aspects of class groups of non-CM extensions of $\Q$. 

One does have at least a partial understanding of the structure of $p$-parts of class groups of $p$-ramified $\F_p$-extensions of $\Q(\mu_p)$ through known values of cup products of cyclotomic $p$-units, and in certain instances one can give lower bounds on $p$-ranks of these groups (see \cite[Section 7]{sh-massey}).  In Section \ref{sec:cyclo}, we consider more complex extensions, deriving lower bounds on the $p$-ranks of class groups of $p$-ramified bicyclic and Heisenberg extensions of $\Q(\mu_p)$ in cases where standard genus theory does not produce any unramified extensions. The key tools in this work are Theorem A, our descriptions of the generalized Bockstein maps $\Psi^{(n)}$ for $n \in \{1,2\}$, and computations of cup products of cyclotomic units from \cite{mcs}. 

We consider the case that the class group of $\Q(\mu_p)$ has $p$-rank $1$. 
Suppose we have an $\F_p^2$-extension $K$ of $\Q(\mu_p)$ that is Galois over $\Q$ for which the cup product pairing with the Kummer cocycle of the Kummer generators of the $\F_p^2$-extension vanish. Under certain assumptions on the action of $\Gal(\Q(\mu_p)/\Q)$ on these Kummer generators, we can show that the $p$-rank of the class group of $K$ is at least $6$ (see Proposition \ref{prop:bicycbound}).  This $\F_p^2$-extension $K$ is then further contained in a Heisenberg extension $L$ of $\Q(\mu_p)$ of degree $p^3$ that is Galois over $\Q$, and the $p$-rank of its class group is at least $7$ (see Proposition \ref{prop:heisbound}). The smallest irregular prime $p$ for which there exist $\F_p^2$-extensions for which these lower bounds are shown to hold by our methods is $101$.

In \cite{reciprocity}, the results of this paper are applied in the setting of Iwasawa theory to study inverse limits of class groups.  There, $G$ is the Galois group of the maximal extension of a number field unramified outside a finite set of primes
containing those above $p$, and $H$ is the Galois group of a $\zp$-extension.  The group $H^2_\Iw(N,\Z_p(1))$ is closely related to,
but not always isomorphic to, the inverse limit $X$ of $p$-parts of class groups under norm maps in the tower of number fields defined by $H$.  The isomorphisms of Theorem A are then used to derive exact sequences describing the graded pieces in the augmentation filtration of $X$.

\subsection{Absolute Galois groups and Massey vanishing}
\label{introAbs}
Let $G$ be a profinite group and let $p$ be a prime number.  Let $\chi \in H^1(G,\F_p) = \Hom(G,\F_p)$.  Consider the sequence
\begin{equation}
\label{eq:cor cup res}
H^1(\ker \chi, \F_p) \xrightarrow{\mathrm{cor}} H^1(G,\F_p) \xrightarrow{\chi\,\cup} H^2(G,\F_p) \xrightarrow{\mathrm{res}} H^2(\ker \chi,\F_p).
\end{equation}
If $G=G_F$ is the absolute Galois group of a field $F$ that contains a primitive $p$th root of unity, then this sequence is exact, as can be seen using the properties of the norm residue symbol. This exactness is an important property of absolute Galois groups: for example, it is used heavily in the proof of the norm residue isomorphism theorem (see \cite{voevodsky}). 

Using Theorem B, we show that:
\begin{itemize}
\item[(i)]  The sequence \eqref{eq:cor cup res} is exact at $H^1(G,\F_p)$ if and only if all $p$-fold Massey products of the form $(\chi^{(p-1)},\lambda)$ with $\chi \cup \lambda=0$ vanish for some proper defining system.
\item[(ii)] If \eqref{eq:cor cup res} is exact at $H^2(G,\F_p)$, then it is exact.
\end{itemize}
In light of (i),  we say that a group $G$ has the \emph{$p$-cyclic Massey vanishing property} if the sequence \eqref{eq:cor cup res} is exact at $H^1(G,\F_p)$ for every $\chi \in H^1(G,\F_p)$.  We prove the following in Theorem \ref{thm:triple vanishing}. 

\begin{thmC}
Let $G$ be a profinite group with the $p$-cyclic Massey vanishing property for an odd prime $p$. Then every Massey triple product on $H^1(G,\F_p)$ which is defined contains zero.
\end{thmC}

If $F$ is a field containing a primitive $p$th root of unity, then its absolute Galois group $G_F$ has the $p$-cyclic Massey vanishing property.  Hence Theorem C implies that every Massey triple product on $H^1(G_F,\F_p)$ which is defined contains zero.  This is the \emph{triple Massey vanishing} theorem of Efrat--Matzri \cite{em} and Min\'a\v{c}--T\^an \cite{mt2} for odd $p$ (which implies the vanishing for arbitrary fields as in the latter paper). For more discussion about absolute Galois groups and the general Massey vanishing conjecture of \cite{mt1}, see the introduction to Section \ref{masseyvanish}.

In our proof of Theorem C, to show that a defined Massey product $(\chi,\lambda,\psi)$ vanishes, we consider the coimage $H$ of the map $(\chi,\psi) \colon G \to \F_p^2$, let $\Omega=\F_p[H]$, and let $I \subset \Omega$ be the augmentation ideal. We then apply a variant of Theorem B to this $H$ to see that the Massey product $(\chi,\lambda,\psi)$ relative to a certain defining system is the obstruction to lifting $\lambda$ to a class in $H^1(G,\Omega/J)$ for a particular ideal $J$ between $I^2$ and $I^3$. 
Via an involved diagram chase, we see that the $p$-cyclic Massey vanishing property for the quotients of $H$ that are the coimages of $\chi$, $\psi$, and $\chi+\psi$ implies that this obstruction equals $\nu \cup (\chi+\psi)$ for some $\nu \in H^1(G,\F_p)$. This is enough to show that the Massey product contains zero.

Theorem C raises several interesting questions that we do not attempt to address here, including whether or not the vanishing of Massey products $(\chi^{(n)},\psi)$ for arbitrary $n$ is sufficient to imply Massey vanishing. 

\section{Generalized Bockstein maps}
\label{sec:bock} 
In this section, we define generalized Bockstein maps and employ them in the study of the structure of inverse limits of cohomology groups.  Throughout the paper, we use the following
objects:
\begin{itemize}
	\item a prime number $p$,
	\item a profinite group $G$,
	\item a topologically finitely generated pro-$p$ quotient $H$ of $G$ by a closed normal subgroup $N$,
	\item a compact Noetherian $\zp$-algebra $R$ (usually taken to be a quotient of $\zp$),
	\item the completed group ring $\Omega = R\ps{H}$,
	\item the augmentation ideal $I$ of $\Omega$, i.e., the kernel of the continuous $R$-algebra homomorphism 
	$\Omega \to R$ that sends every group element in $H$ to $1$,
	\item a positive integer $n$ such that $\Omega/I^n$ and $I^n/I^{n+1}$ are $R$-flat, and
	\item a compact $R\ps{G}$-module $T$ that is $R$-finitely generated.
\end{itemize}

Note that a compact $R\ps{G}$-module is the same as a compact $R$-module with a continuous $R$-linear action of $G$.  We will frequently take tensor products $M \otimes_R M'$ of compact $R\ps{G}$-modules $M$ and $M'$, at least one of which is finitely generated over $R$.  These compact $R$-modules (with the topology of the isomorphic completed tensor product) have the diagonal action of $G$.

We are concerned in this paper with the continuous cohomology groups $H^i(G,M)$ of compact $R\ps{G}$-modules $M$ for $i \ge 0$. In particular, $G$-cochains are implicitly supposed to be continuous. We use use square brackets to denote both classes of cocycles and group elements in completed group algebras, and we denote an element in a module and its coset in a quotient thereof by the same symbol where the context is clear.

\subsection{Augmentation sequences}

Since we have assumed that $\Omega/I^n$ is $R$-flat, the right exact sequence of compact $R\ps{G}$-modules
\begin{equation} \label{bocksteinT}
	0 \to 	T \otimes_R I^n/I^{n+1} \to T \otimes_R \Omega/I^{n+1}  \to T \otimes_R \Omega/I^n  \to 0
\end{equation} 
is exact.
For any $d \ge 1$, we have the resulting connecting homomorphisms 
\[
	H^{d-1}(G,T \otimes_R \Omega/I^n) \to H^d(G,T \otimes_R I^n/I^{n+1})
\]
on continuous $G$-cohomology.

Since $G$ acts trivially on the finitely generated $R$-module $I^n/I^{n+1}$, we have a homomorphism
\begin{equation} \label{moveout}
 H^d(G,T) \otimes_R I^n/I^{n+1} \to H^d(G,T \otimes_R I^n/I^{n+1})
\end{equation}
that is an isomorphism as $I^n/I^{n+1}$ is $R$-flat, so long as we assume either that $G$ has finite $p$-cohomological
dimension or that $I^n/I^{n+1}$ has a finite resolution by projective $R$-modules (see \cite[Proposition 3.1.3]{ls}, the proof of which
does not use the assumption on $R$ in that section).  The latter condition is automatic, given that $I^n/I^{n+1}$ is flat, 
if $R$ is a quotient of $\zp$.
We let
\begin{equation} \label{Psin}
	\Psi^{(n)} \colon H^{d-1}(G,T \otimes_R \Omega/I^n) \to H^d(G,T) \otimes_R I^n/I^{n+1}
\end{equation}
denote the resulting composite map, and we refer to it as a \emph{generalized Bockstein map}.  

\begin{remark}
	We may replace the assumption that $\Omega/I^n$ is $R$-flat with the assumption that $T$ is $R$-flat
	in order that \eqref{bocksteinT} still holds.  We may also replace the assumption that $I^n/I^{n+1}$ is $R$-flat with the 
	assumption that $G$ has $p$-cohomological dimension $d$ and still have 
	an isomorphism as in \eqref{moveout}.  (To see this, choose a presentation of $I^n/I^{n+1}$ by finitely generated free
	$R$-modules and use the right exactness of the $d$th cohomology functor and the tensor product, noting that $H^d(G,T^r) \cong
	H^d(G,T) \otimes_R R^r$ for any $r$.)  
	With either replacement, $\Psi^{(n)}$ is still a map as in \eqref{Psin}.
\end{remark}

\subsection{Graded quotients of Iwasawa cohomology groups} \label{descent}

Recall that $N$ denotes the kernel of the surjection $G \to H$.  Our interest in this section is in the \emph{Iwasawa cohomology groups}
$$
	H^i_{\Iw}(N,T) = \varprojlim_{N \leq U \unlhd^o G} H^i(U,T)
$$
for $i \ge 1$,
where the inverse limit is taken with respect to corestriction maps over open normal subgroups $U$ of $G$ containing $N$.  
Note that the Iwasawa cohomology groups are relative to the larger group $G$, though this is omitted from our notation.
Since each $H^i(U,T)$ is a $R[G/U]$-module and the actions are compatible with corestriction, 
the group $H^i_{\Iw}(N,T)$ is endowed with the structure of an $\Omega$-module.  

\begin{remark}
	If $H$ is finite, then $H^i_{\Iw}(N,T) = H^i(N,T)$.
\end{remark}

Let us define two notions that we need.  First, a profinite group $\mc{G}$ is \emph{$p$-cohomologically finite} if $\mc{G}$ has finite $p$-cohomological dimension and  $H^i(\mc{G},M)$ is finite for every finite $\zp[\mc{G}]$-module $M$ and $i \ge 0$. 
Second, a \emph{compact $p$-adic Lie group} is a profinite group that has an open pro-$p$ subgroup, any closed subgroup of which can be topologically generated by $r$ elements for some fixed $r$.  Equivalently, a compact $p$-adic Lie group is any profinite group continuously isomorphic to a closed subgroup of $\GL_n(\zp)$ for some $n \ge 1$.

We make the following assumptions for the rest of this section:
\begin{itemize}
	\item $G$ is $p$-cohomologically finite of $p$-cohomological dimension $d$,
	\item $R$ is a complete commutative local Noetherian $\zp$-algebra with finite residue field, and
	\item either 
	\begin{enumerate}
		\item[(i)] $H$ is a compact $p$-adic Lie group or 
		\item[(ii)] $T$ has a finite resolution by a complex of $R\ps{G}$-modules free 
		of finite rank over $R$.
	\end{enumerate}
\end{itemize}

Recall that the zeroth $H$-homology group of a compact $\Omega$-module $M$ is its coinvariant module $M_H \cong M/IM$.  
In our setting, corestriction provides an isomorphism on coinvariants in degree $d$ (see \cite[Proposition 3.3.11]{nsw}),
which is to say that we have a natural isomorphism
\begin{equation} \label{coinvisom}
	\frac{H^d_{\Iw}(N,T)}{IH^d_{\Iw}(N,T)} \cong H^d(G,T).
\end{equation}
This gives rise to a Grothendieck spectral sequence for the implicit composition of right exact functors, which is a version of Tate's descent spectral sequence for Iwasawa cohomology.

\begin{proposition}[Fukaya-Kato, Lim-Sharifi] \label{spectral}
	The $\Omega$-modules $H^i_{\Iw}(N,T)$ are finitely generated for all $i \ge 0$.
	Moreover, we have a first quadrant homological spectral sequence of $R$-modules
	$$
		E^2_{i,j}(T) = H_i(H,H^{d-j}_{\Iw}(N,T)) \Rightarrow E_{i+j}(T) = H^{d-i-j}(G,T),
	$$
	where $d$ is the $p$-cohomological dimension of $G$.
\end{proposition}

This result is proven in \cite[Theorem 1]{tate} if $H$ is finite, and it follows from \cite[Proposition 1.6.5]{fk} if (ii) holds and from \cite[Propositions 3.1.3 and 3.2.4]{ls} if (i) holds. 

The isomorphism \eqref{coinvisom} and the other edge maps on coinvariant groups in this spectral sequence are given by the inverse limits of corestriction maps.  This isomorphism forces the $n$th graded quotient $I^nA/I^{n+1}A$ in the augmentation filtration of $A = H^d_{\Iw}(N,T)$ to be a quotient of $H^d(G,T) \otimes_R I^n/I^{n+1}$ using the surjective  map
$$
	A/IA \otimes_R I^n/I^{n+1} \to I^nA/I^{n+1}A
$$
induced by the map $A \times I^n \to I^nA$ given by the multiplication $(a,x) \mapsto xa$.
As we shall see, this quotient is in fact $\coker \Psi^{(n)}$.

Recall that we have assumed that $\Omega/I^n$ is $R$-flat.  Moreover, the fact that $H$
is topologically finitely generated implies that $\Omega/I^n$ is finitely generated over $R$.

\begin{lemma} \label{bocklemma}
	Let $A$ be an $\Omega$-module, and consider the exact sequence
	\begin{equation} \label{bockstein}
		0 \to A \otimes_R I^n/I^{n+1} \to 
		A \otimes_R \Omega/I^{n+1} \to 
		A \otimes_R \Omega/I^n  \to 0.
	\end{equation} 
	The connecting homomorphism
	$$
		\partial_n \colon H_1(H,A \otimes_R \Omega/I^n) \to A_H \otimes_R I^n/I^{n+1}
	$$
	in the $H$-homology of \eqref{bockstein}
	has cokernel isomorphic to $I^n A/I^{n+1} A$. 
\end{lemma}

\begin{proof}
	We have compatible, natural isomorphisms of $R$-modules 
	$$
		(A \otimes_R \Omega/I^m)_H \cong \Omega/I^m \otimes_{\Omega} A \cong A/I^mA
	$$
	for $m \ge 1$ given on $a \in A$ and $\omega \in \Omega$ (or its quotient by $I^m$) by 
	$$
		a \otimes \omega \mapsto \iota(\omega) \otimes a \mapsto \iota(\omega)a,
	$$
	where $\iota \colon \Omega \to \Omega$ is
	the unique continuous $R$-linear map given by inversion of group elements on $H$.
	Note that the switch of terms in the tensor product in the first isomorphism is necessitated by the fact that $A$ is a left $\Omega$-module.
	(In fact, these become isomorphisms of $\Omega$-modules since
	$a \otimes \omega h^{-1} \mapsto h \cdot \iota(\omega) \otimes a$ for $h \in H$
	under the first map.)
	
	By the long exact sequence in $H$-homology and the above isomorphisms, the cokernel of interest is 
	identified with the kernel of the quotient map $A/I^{n+1}A \to A/I^nA$, hence the result.
\end{proof}

We now come to our theorem.

\begin{theorem} \label{augfilt}
For each $n \ge 1$, there is a canonical isomorphism
$$
	\frac{I^n H^d_{\Iw}(N,T)}{I^{n+1} H^d_{\Iw}(N,T)} \cong 
	\frac{H^d(G,T) \otimes_R I^n/I^{n+1}}{\image \Psi^{(n)}}
$$
of $R$-modules, where $d$ is the $p$-cohomological dimension of $G$.
\end{theorem}

\begin{proof}
    	There are isomorphisms 
        $$
        	H^d_{\Iw}(N,T \otimes_R M) \cong H^d_{\Iw}(N,T) \otimes_R M
        $$
        for any compact $R\ps{G}$-module $M$ finitely generated over $R$, since $G$ has $p$-cohomological dimension $d$.
        In particular, the following sequence is exact:
        $$
        	0 \to H^d_{\Iw}(N,T \otimes_R I^n/I^{n+1})  \to H^d_{\Iw}(N,T \otimes_R \Omega/I^{n+1}) \to
        	H^d_{\Iw}(N,T \otimes_R \Omega/I^n) \to 0.
        $$
        We consider the connecting homomorphism in $H$-homology:
        \begin{equation} \label{partial}
        	\partial^{(n)} \colon H_1(H,H^d_{\Iw}(N, T) \otimes_R \Omega/I^n)
        	\to H^d_{\Iw}(N,T)_H \otimes_R I^n/I^{n+1}.
        \end{equation} 
        
        We next apply Lemma \ref{commsquare} of the appendix, which says that edge maps to total terms in homological
        Grothendieck spectral sequences are compatible with connecting maps.  Here, the spectral sequence is that of
        Proposition \ref{spectral}, which is associated to the composition of functors $F = H_0(H, \sdot)$ and 
        $F' = H^d_{\Iw}(N, \sdot)$, noting that $F \circ F' \cong H^d(G,\sdot)$ via corestriction.  The connecting homomorphisms are 
        from degrees $1$ to $0$ and are associated to the short exact sequence of \eqref{bocksteinT}.
      
	In this setting, the lemma provides a commutative square related to the diagram
	\begin{equation} \label{keysquare}
		\begin{tikzcd}
			H^{d-1}(G,T \otimes_R \Omega/I^n) \arrow{r}{\Psi^{(n)}} \arrow[d, twoheadrightarrow]
			& H^d(G,T) \otimes_R I^n/I^{n+1} \arrow{d}{\wr} \\
			H_1(H,H^d_{\Iw}(N, T) \otimes_R \Omega/I^n) \arrow{r}{\partial^{(n)}} & H^d_{\Iw}(N,T)_H \otimes_R I^n/I^{n+1},
		\end{tikzcd}
	\end{equation}
	but with $L_1(F \circ F')(T \otimes_R \Omega/I^n)$ in place of $H^{d-1}(G,T \otimes_R \Omega/I^n)$.  
	By Lemma \ref{dontneedderived}, which is a simple consequence of the universality of left derived functors, we have a
	surjection from the latter object to the former, compatible with their connecting homomorphisms to 
	$H^d(G,T) \otimes_R I^n/I^{n+1}$. 
	This allows us to make the replacement while maintaining the surjectivity of the left vertical map, so we indeed
	have the commutative square \eqref{keysquare}.  
	
	By Lemma \ref{bocklemma}, the isomorphism in the statement 
	of the theorem is the map on cokernels of the horizontal maps in \eqref{keysquare}.
\end{proof}

Although not used in this paper, for the purposes of Iwasawa-theoretic applications, it is useful to have a slightly stronger version of 
Theorem \ref{augfilt}.  So, we remark that it has the following generalization, with virtually no additional complications
(given that the results of \cite{ls, fk} hold in this generality).

\begin{remark}
     Let $\mc{G}$ be a profinite group, and let $\Gamma$ be a quotient of $\mc{G}$ by a closed normal subgroup
    $G$.  
    Let $\mc{H}$ be a quotient of $\mc{G}$ by a closed normal subgroup $N$ that is
    contained in $G$, and let $H = G/N$ as before.  
    We then have $\Gamma \cong \mc{H}/H$.  That is, we have a commutative diagram of exact sequences
    $$
    	\begin{tikzcd}
    		N \arrow[d,hook] \arrow[r,equal] & N \arrow[d,hook] \\
    		G \arrow[r,hook] \arrow[d,two heads] & \mc{G} \arrow[r, two heads] \arrow[d, two heads] & \Gamma \arrow[d,equal] \\
    		H \arrow[r,hook] & \mc{H} \arrow[r,two heads] & \Gamma.
    	\end{tikzcd}
    $$
    
    Take $T$ to be a compact $R\ps{\mc{G}}$-module finitely generated over $R$, and 
    replace the assumptions on $G$ and $H$ from the beginning of this subsection with the identical assumptions
    on $\mc{G}$ and $\mc{H}$, respectively.  We have 
    Iwasawa cohomology groups $H^i_{\Iw}(N,T)$ and $H^i_{\Iw}(G,T)$, which are now taken relative to the larger group $\mc{G}$.  These 
    are finitely generated as modules over $R\ps{\mc{H}}$ and $\Lambda = R\ps{\Gamma}$, respectively, and we have as before 
    a spectral sequence
    $$
    	E^2_{i,j}(T) = H_i(H,H^{d-j}_{\Iw}(N,T)) \Rightarrow E_{i+j}(T) = H^{d-i-j}_{\Iw}(G,T),
    $$
    but now of $\Lambda$-modules.  In exactly the same manner as before, this gives rise to isomorphisms
    $$
    	\frac{I^n H^d_{\Iw}(N,T)}{I^{n+1} H^d_{\Iw}(N,T)} \cong \frac{H^d_{\Iw}(G,T) \otimes_R I^n/I^{n+1}}{\image \Psi^{(n)}},
    $$
    again of $\Lambda$-modules.
\end{remark}

\subsection{The abelian case} \label{abelian}

We turn to the direct computation of generalized Bockstein maps on $1$-cocycles for abelian $H$.  That is, let us now take $H$ to be a finitely generated, abelian pro-$p$ group, and
let us take $R$ to be a quotient of $\zp$.   We give an explicit formula for $\Psi^{(n)}$ under a hypothesis on the size of $R$ that ensures our flatness hypothesis is satisfied.  If $H$ has no nonzero $p$-torsion, no hypothesis is needed.

We begin with the following simple lemma.

\begin{lemma}
\label{lem:binomial}
Let $s$ and $t$ be positive integers with $n < p^{t-s+1}$.  Then $(1+x)^{p^t}-1$ is in the ideal $(x^{n+1},p^s)$ of $\Z[x]$.
\end{lemma}

\begin{proof}
Recall that $p^s$ divides $\binom{p^t}{i}$ for $0 < i < p^{t-s+1}$.  
Therefore
$$
	(1+x)^{p^t} = \sum_{i=0}^{p^t} \binom{p^t}{i}x^i \equiv 1 \bmod (x^{n+1},p^s)
$$ 
so long as $n < p^{t-s+1}$.
\end{proof}

Let $h_1, \dots, h_r$ be a minimal set of generators for $H$, labeled such that $h_1, \dots, h_c$ have finite orders $p^{t_1}\le  \dots \le  p^{t_c}$ and $h_{c+1},\dots, h_r$ have infinite order, for some $0 \le c \le r$. Define $x_i = [h_i] - 1 \in \Omega$ for $1 \le i \le r$, where $[h_i]$ denotes the group element of $h_i$, so that $I = (x_1, \dots, x_r)$.  We then have
$$
	\Omega \cong \frac{R\ps{x_1,\ldots,x_r}}{((x_1+1)^{p^{t_1}}-1,\ldots,(x_c+1)^{p^{t_c}}-1)}.
$$
We have $c >0$ if and only if $H$ is not $\zp$-free, in which case we suppose that $R=\Z/p^s\Z$ with $n < p^{t_1-s+1}$. 
By Lemma \ref{lem:binomial}, we have
\[
	\Omega/I^j \cong \frac{R[x_1,\ldots,x_r]}{(x_1,\ldots,x_r)^j}
\]
for $j \le n+1$. Moreover, $I^n/I^{n+1}$ is a free $R$-module with basis the monomials in the variables $x_i$ of degree $n$.  In particular, the generalized Bockstein map $\Psi^{(n)}$ is defined.

We may view any element of $q \in T \otimes_R \Omega/I^n$ as having the form
$$
	q = \sum_{k_1+ \cdots +k_r < n} \alpha_{k_1,\ldots,k_r} x_1^{k_1} \cdots x_r^{k_r},
$$ 
where the sum is taken over $r$-tuples $(k_1, \ldots, k_r)$ of nonnegative integers with sum less than $n$ and
with $\alpha_{k_1,\ldots,k_r} \in T$, omitting the notation for the tensor product in such an expression.  
Setting $\|k\| = k_1 + \cdots + k_r$ for an $r$-tuple $(k_1, \ldots, k_r)$, let's simplify this notation as
\begin{equation} \label{shortform}
	q = \sum_{\|k\| < n} \alpha_k x^k,
\end{equation}
where $x^k = x_1^{k_1} \cdots x_r^{k_r}$.

Let $\pi \colon G \to H$ denote the quotient map.  For each $i$, let
$$
	A_i = \begin{cases} \Z/p^{t_i}\Z & \text{if } 1 \le i \le c, \\ \zp & \text{if } c < i \le r. \end{cases}
$$
For $1 \le i \le r$, let $\chi_i \colon G \to A_i$ be the homomorphisms determined by
$$
	\pi(g) = \prod_{i=1}^r h_i^{\chi_i(g)}
$$ 
for $g \in G$.  
The action of $g \in G$ on $q$ as in \eqref{shortform} is given by multiplication by 
$\prod_{i=1}^r (1+x_i)^{\chi_i(g)}$.  That is, we have the formula
\begin{equation} \label{act_sum}
	g \cdot q =  \sum_{\|k\| < n}
	\left(\sum_{0 \le k' \le k} \binom{\chi(g)}{k'} 
	g\alpha_{k-k'} \right) x^k,
\end{equation}
where the second sum is over $r$-tuples $k'$ of nonnegative integers with $k'_i \le k_i$ for each $i$, and where we have set
$\binom{\chi(g)}{k'} = \binom{\chi_1(g)}{k'_1} \cdots \binom{\chi_r(g)}{k'_r}$.

Note that our assumption on the cardinality of $R$ can be rephrased as saying that either $c=0$ or $R$ is a quotient of $A_1$ such that 
$|R| < \frac{p}{n}|A_1|$. With our notation and this assumption established, we can give an explicit formula for $\Psi^{(n)}$.

\begin{proposition} \label{connhomcomp}
	Let $f \colon G \to T \otimes_R \Omega/I^n$ be a $1$-cocycle, and write 
	$$
		f = \sum_{\|k\| < n} \lambda_k x^k
	$$
	with $\lambda_k \colon G \to T$.   Then $\Psi^{(n)}$ takes the class of $f$ to the class of the $2$-cocycle
	$$
		(g,h) \mapsto \sum_{\|k\| = n} \left( \sum_{0 < k' \le k}
		\binom{\chi(g)}{k'} g \lambda_{k-k'}(h) \right) x^k,
	$$
	where the first sum is taken over $r$-tuples $k = (k_1, \ldots, k_r)$ of nonnegative integers summing to $n$, and
	the second sum is taken over nonzero $r$-tuples $k'$ of nonnegative integers with $k'_i \le k_i$ for all $i$.
\end{proposition}

\begin{proof}
	Consider the set-theoretic section 
	\begin{equation} \label{splitting}
		s_n \colon T \otimes_R \Omega/I^n \to T \otimes_R \Omega/I^{n+1}
	\end{equation}
	that takes a sum as in \eqref{shortform} to the same expression in the larger module.  Let $\tilde{f} = s_n \circ f$.
	By definition, $\Psi^{(n)}([f])$ is the class of $d\tilde{f}$, where
	$$
		d\tilde{f}(g,h) = \tilde{f}(g) + g\tilde{f}(h) - \tilde{f}(gh)
	$$
	for $g, h \in G$.  Since $f$ is a cocycle, the right-hand side of this expression is equal to the degree $n$ part of $g\tilde{f}(h)$,
	which by \eqref{act_sum} is exactly as in the statement of the proposition.
\end{proof}

For general $H$, pro-$p$ but not necessarily abelian, we can use this computation to see that $\Psi^{(1)}$ is given by cup products.  We consider the case that $H$ is a quotient of $G$ such that the abelianization $H^\mathrm{ab}$ of $H$ is finitely generated and pro-$p$. 
As before, but now for $H^{\ab}$ in place of $H$, there are nonnegative integers $r \ge c$ and positive integers $t_1 \le  \cdots \le t_c$ such that 
\begin{equation} \label{Hab}
	H^\mathrm{ab} \cong \bigoplus_{i=1}^r A_i,
\end{equation}
where $A_i = \Z/p^{t_i}\Z$ for $i=1,\dots,c$ and $A_i=\Z_p$ for $i=c+1,\dots, r$.  For $i=1,\dots, r$, we let $\chi_i \colon G \to A_i$ denote the 
quotient map $G \to A_i$.   We take $n= 1$, and our condition on the cardinality of $R$ becomes
$s \le t_1$ when $c \ge 1$.

Fix generators $h_1, \ldots, h_r$ of $H$ such that each $h_i$ maps to $1 \in A_i$ under the composition of the quotient map and the isomorphism in \eqref{Hab}.  There is an isomorphism $I/I^2 \cong H^{\ab} \otimes_{\zp} R$ taking the image of $x_i = [h_i]-1$ to $h_i \otimes 1$.

\begin{proposition}
\label{prop:abelian}
	Let $H$ be a finitely generated pro-$p$ group with $H^{\ab}$ as in \eqref{Hab}, let $I$ be the augmentation ideal in $\Omega = R\ps{H}$,
	and let $\chi_i$ and $x_i$ for $1 \le i \le r$ be
	as in the previous paragraph.  For any $1$-cocycle $f \colon G \to T$, we have
    	\[
    		\Psi^{(1)}([f]) = \sum_{i=1}^r (\chi_i \cup f) x_i \in H^2(G,T) \otimes_R I/I^2.
    	\]
\end{proposition}

\begin{proof}
    Let $\Omega'=R\ps{H^\mathrm{ab}}$ with augmentation ideal $I' \subset \Omega'$. 
    Both $\Omega/I$ and $\Omega'/I'$ are identified with $R$ via the augmentation maps, and there are also compatible
    isomorphisms between the graded quotients $I/I^2$ and $I'/(I')^2$ and the $R$-module $H^{\ab} \otimes_{\zp} R$.  It follows that
    the canonical map $\Omega \to \Omega'$ induces an isomorphism $\Omega/I^2 \cong \Omega'/(I')^2$.
    Thus $\Psi^{(1)}$ equals the first generalized Bockstein map for $H^\mathrm{ab}$, and the proposition 
    follows from the case $n=1$ of Proposition  \ref{connhomcomp}.
\end{proof}

This result was previously studied by the third author in the context of Iwasawa theory, where these maps are referred to as
reciprocity maps with restricted ramification: see for instance \cite[Lemma 4.1]{sh-paireis} for its introduction.
In the following section, we study analogous results for $\Psi^{(n)}$ with $n > 1$ in terms of higher Massey products.

\section{Massey products}\label{sec:masseyreview}
In this section, we review the definitions of Massey products and defining systems, with some modifications from the standard definitions in order to allow for nontrivial coefficient modules.  We also introduce the notions of partial and proper defining systems.

\subsection{Upper-triangular generalized matrix algebras}
The notion of Massey products that we will use is conveniently stated using the theory of generalized matrix algebras, as found in \cite{bc}.  We require only a simple upper-triangular version of these algebras. Let $n$ be a positive integer, and let $R$ be a commutative ring.

\begin{definition} \label{UGMA}
An $n$-dimensional \emph{upper-triangular generalized matrix algebra} $\mc{A}$ over $R$ (or, \emph{$R$-UGMA}) is an $R$-algebra
formed out of the data of
\begin{itemize}
\item finitely generated $R$-modules $A_{i,j}$ for $1 \le i \le j \le n$ with $A_{i,j} = R$ if $i=j$, and
\item $R$-module homomorphisms $\varphi_{i,j,k} \colon A_{i,j} \otimes_R A_{j,k} \to A_{i,k}$ for all $1 \le i \le j \le k \le n$
which are induced by the given $R$-actions if $i = j$ or $j = k$
\end{itemize}
such that the two resulting maps 
$$
	A_{i,j} \otimes_R A_{j,k} \otimes_R A_{k,l} \to A_{i,l}
$$ 
coincide for all $1 \le i < j < k < l \le n$.  The tuple $(A_{i,j},\varphi_{i,j,k})$ defines an $R$-algebra $\mc{A}$ with underlying $R$-module
$$
	\mc{A}=\bigoplus_{1 \le i \le j \le n} A_{i,j},
$$ 
and multiplication given by matrix multiplication: that is, for $a = (a_{i,j})$ and $b = (b_{i,j})$ in $\mc{A}$, the $(i,j)$-entry $(ab)_{i,j}$ of $ab$ is 
\[
(ab)_{i,j} = \sum_{k=i}^j \varphi_{i,k,j}(a_{i,k} \otimes b_{k,j}).
\]
\end{definition}

Our interest is in the multiplicative group $\mc{U} = \mc{U}(\mc{A})$ of \emph{unipotent} matrices in a UGMA $\mc{A}$, i.e., those $a = (a_{i,j})$ with $a_{i,i} = 1$ for all $i$.
We shall often take the quotient $\mc{U}' = \mc{U}'(\mc{A})$ of this $\mc{U}$ by its central subgroup $\mc{Z} = \mc{Z}(\mc{A})$ of \emph{unipotent central} elements, i.e., those $a \in \mc{U}$ with $a_{i,j} = 0$ for all $(i,j) \neq (1,n)$.

The following is the key example for our purposes.
\begin{example}
\label{eg:GMA we use}
Let $M$ be a finitely generated $R$-module, and let $m$ be a positive integer less than $n$. We define an $n$-dimensional $R$-UGMA 
$\mc{A}_n(M,m)$ as follows.  Set 
\[
	A_{i,j} = \begin{cases} 
	M &\text{if }i \le m < j, \\
	R &\text{ otherwise}\end{cases}
\]
and take the maps $\varphi_{i,j,k}$ to be the $R$-module structure maps.  This makes sense since, given $i \le j \le k$, at least one of $A_{i,j}$ and $A_{j,k}$ must be $R$, as $m$ cannot satisfy both $m < j$ and $j \le m$.

Let us write $\mc{U}_n(M,m)$ for $\mc{U}(\mc{A}_n(M,m))$ and $\mc{U}'_n(M,m)$ for $\mc{U}'(\mc{A}_n(M,m))$.
To make this easier to visualize, note that we can write $\mc{U}_n(M,m)$ in ``block matrix'' form as
$$
	\mc{U}_n(M,m) = \begin{pmatrix} \mr{U}_m(R) & \mr{M}_{m,n-m}(M) \\ 0 & \mr{U}_{n-m}(R) \end{pmatrix},
$$
where $\mr{U}_k(R) \leqslant \GL_k(R)$ denotes the group of upper-triangular unipotent matrices 
and $\mr{M}_{k,l}(M)$ denotes the additive group of $k$-by-$l$ matrices with entries in $M$ for positive integers $k$ and $l$.
The latter group is endowed with a left $\mr{U}_k(R)$-action and a commuting right $\mr{U}_l(R)$-action.  Put differently, $\mc{A}_n(M,m)$
itself is a sort of $2$-by-$2$ generalized matrix algebra, allowing noncommutative rings on the diagonal and bimodules in the
non-diagonal entries.
\end{example}

We actually need to use \emph{profinite UGMAs} defined just as in Definition \ref{UGMA} using profinite rings $R$ and compact $R$-modules $A_{i,j}$, but now assuming that the induced multiplication maps $A_{i,j} \times A_{j,k} \to A_{i,k}$ are continuous.  Alternatively, the maps $\varphi_{i,j,k}$ can be replaced by maps of completed tensor products over $R$ in the definition.

Though unnecessary, to keep things simple, let us suppose that the compact $R$-modules $A_{i,j}$ in a profinite $R$-UGMA are $R$-finitely generated.  This forces them to have the adic topology for any directed system of ideals that are open neighborhoods of zero.  Moreover, their tensor products and completed tensor products are then abstractly isomorphic, and so we may in particular view the tensor products $A_{i,j} \otimes_R A_{j,k}$ themselves as compact $R$-modules.  (For a slightly longer discussion of this, see \cite[Section 2.3]{ls}.)

Note that any profinite $R$-UGMA $\mc{A}$ has a topology as a finite direct product of the compact $R$-modules $A_{i,j}$, and $\mc{U}$ inherits the subspace topology.

We also want to make a second modification, allowing a continuous action of $G$.

\begin{definition} For a profinite ring $R$ and a profinite group $G$, a \emph{profinite $(R,G)$-UGMA} is the data of a profinite $R$-UGMA $\mc{A}$ together with a continuous $G$-action on each $A_{i,j}$ such that
\begin{itemize}
\item the action on $A_{i,i} = R$ is trivial for all $i$, and
\item the maps $\varphi_{i,j,k}$ are maps of $R\ps{G}$-modules, where $A_{i,j} \otimes_R A_{j,k}$ is given the diagonal action of $G$.
\end{itemize}
\end{definition}

We remark that, aside from issues of finite generation, the difference between a profinite $R\ps{G}$-UGMA and a profinite $(R,G)$-UGMA is that in the former, each $A_{i,i} = R\ps{G}$, whereas in the latter, each $A_{i,i}$ is $R$ with the trivial $G$-action.  We are interested in the latter structure.  

\begin{example}
\label{eg:G GMA we use}
If $R$ is a profinite ring and $T$ is a compact $R\ps{G}$-module (that is $R$-finitely generated), then the $R$-UGMA $\mc{A}_n(T,m)$ of Example \ref{eg:GMA we use} has a natural structure of a profinite $(R,G)$-UGMA by letting $G$ act on $A_{i,j}$ via its action on $T$ if $i \le m < j$ and trivially otherwise.
\end{example}

\subsection{Defining systems and Massey products}\label{section:massdefn}

Let $R$ be a profinite ring, let $G$ be a profinite group, and let $n \ge 2$.  Let $T_1, \dots, T_n$ be compact $R\ps{G}$-modules that are $R$-finitely generated for simplicity, and let $\chi_i \colon G \to T_i$ be continuous 1-cocycles for $1 \le i \le n$. In this section, we define \emph{Massey products} of these cocycles, which will be 2-cocycles that depend on a number of choices constituting a \emph{defining system}.

\begin{definition}
\label{defn:def sys}
A \emph{defining system} for the Massey product of $\chi_1,\dots,\chi_n$ is the data of
\begin{itemize}
    \item an $(n+1)$-dimensional profinite $(R,G)$-UGMA $\mc{A}$ and
    \item a (nonabelian) continuous 1-cocycle $\rho \colon G \to \mc{U}'$
\end{itemize}
such that $A_{i,i+1}=T_i$ for $1 \le i \le n$, and the composition of $\rho$ with projection to $A_{i,i+1}$ is $\chi_i$.
\end{definition}

Given a defining system $\rho \colon G \to \mc{U}'$, there is a unique function $\tilde{\rho} \colon G \to \mc{U}$ lifting $\rho$ and having zero as the $(1,n+1)$-entry of $\tilde{\rho}(g)$ for all $g \in G$. We let $\rho_{i,j} \colon G \to A_{i,j}$ be the map given by taking the $(i,j)$-entry of $\tilde{\rho}$.

\begin{definition}
\label{defn:massey product}
Given a defining system $\rho$, the \emph{$n$-fold Massey product} $(\chi_1,\dots,\chi_n)_\rho \in H^2(G,A_{1,n+1})$ is the class of the 2-cocycle
\[
(g,h) \mapsto \sum_{i=2}^n \varphi_{1,i,n+1}(\rho_{1,i}(g) \otimes g \rho_{i,n+1}(h))
\]
that sends $(g,h)$ to the $(1,n+1)$-entry of $\tilde{\rho}(g)\cdot g \tilde{\rho}(h)$.
\end{definition}

In the remainder of the paper, we will restrict our attention to the setting of the $(n+1)$-dimensional profinite $(R,G)$-UGMAs of the form $\mc{A}_{n+1}(T,m)$ defined in Examples \ref{eg:GMA we use} and
\ref{eg:G GMA we use}.  This means in particular that we only consider $n$-fold Massey products for which there is an $m$ with
$1 \le m \le n$ such that $T_m = T$ and $T_i = R$ for $i \neq m$. In particular, we will always have $(\chi_1,\dots,\chi_n)_\rho \in H^2(G,T)$.

In \cite{sh-massey}, the third author considered the case in which $m = n$ and $\chi_1 = \cdots = \chi_{n-1}$ in a Galois-cohomological setting.  In that case, the key idea 
for relating Massey products to graded pieces of Iwasawa cohomology groups
was to consider only a restricted set of defining systems referred to as \emph{proper defining systems}. We will consider a more general notion of proper defining system that depends on extra data we call a \emph{partial defining system}. In \cite{sh-massey}, the partial defining system comes from unipotent binomial matrices, which we review in Section \ref{subsec:unipotent binomials} below.

\subsection{Massey products relative to proper defining systems}

Fix an integer $n \ge 2$ and two integers $a,b \ge 0$ with $a+b=n$.  Let $Z^i(G,M)$ for a profinite $R\ps{G}$-module $M$ denote the group of continuous $i$-cocycles on $G$ valued in $M$.  Choose tuples 
\begin{eqnarray*}
	\alpha=(\alpha_1,\dots,\alpha_a) \in Z^1(G,R)^a &\mr{and}& \beta = (\beta_1,\dots,\beta_b) \in Z^1(G,R)^b
\end{eqnarray*} 
and a compact $R\ps{G}$-module $T$ that is finitely generated as an $R$-module.  

We next consider a pair of homomorphisms that constitute a part of the defining systems for $(n+1)$-fold Massey products $(\alpha_1, \ldots, \alpha_a, \lambda, \beta_1, \ldots, \beta_b)$, where $\lambda \in Z^1(G,T)$ is allowed to vary.  We write the collection of such Massey products as $(\alpha,\sdot,\beta)$ for short.

\begin{definition}
    A \emph{partial defining system} for $(n+1)$-fold Massey products $(\alpha, \sdot, \beta)$
    is a pair of homomorphisms
    \begin{eqnarray*}
    	\phi \colon G \to \mr{U}_{a+1}(R) &\mr{and}& \theta \colon G \to \mr{U}_{b+1}(R)
    \end{eqnarray*}
    such that $\alpha$ is the off-diagonal of $\phi$ and $\beta$ is the off-diagonal of $\theta$, i.e.,
    $\phi_{i,i+1} = \alpha_i$ for $1 \le i \le a$ and $\theta_{i,i+1} = \beta_i$ for $1 \le i \le b$.
    
    More specifically, an \emph{$(a,b)$-partial defining system} is a partial defining system restricting to some pair $(\alpha,\beta)
    \in Z^1(G,R)^a \times Z^1(G,R)^b$.
\end{definition}

Recall that
$$
	\mc{U}_{n+2}(T,a+1) = \begin{pmatrix} \mr{U}_{a+1}(R) & M_{a+1,b+1}(T) \\ & \mr{U}_{b+1}(R) \end{pmatrix}.
$$
We may then write the quotient by the unipotent central matrices as
$$
	\mc{U}'_{n+2}(T,a+1) = \begin{pmatrix} \mr{U}_{a+1}(R) & M'_{a+1,b+1}(T) \\ & \mr{U}_{b+1}(R) \end{pmatrix}
$$
for $M'_{a+1,b+1}(T) = M_{a+1,b+1}(T)/T$, where $T$ is identified with the matrices that are zero outside the $(1,b+1)$-entry.

\begin{definition}  \label{proper_def}
    Given a $1$-cocycle $\lambda \colon G \to T$,
    a \emph{proper defining system} for an $(n+1)$-fold Massey product $(\alpha,\lambda,\beta)$
    \emph{relative to a partial defining system $(\phi,\theta)$} is a continuous $1$-cocycle
    \[
    	\rho \colon G \to \mc{U}'_{n+2}(T,a+1)
    \]
    of the form
    \[
    	\rho = \ttmat{\phi}{\kappa}{0}{\theta}
    \]
    for some $\kappa \colon G \to M'_{a+1,b+1}(T)$ with $\kappa_{a+1,1} = \lambda$.
\end{definition}

The advantage of proper defining systems is that they are parameterized by abelian, rather than nonabelian, cocycles. To show this, we  introduce a compact $R\ps{G}$-module $\mf{U}_{\phi,\theta}(T)$ such that proper defining systems in $T$ relative to $(\phi,\theta)$ correspond to 1-cocycles with values in $\mf{U}_{\phi,\theta}(T)$.

Consider the compact $R$-module $\mathfrak{U}_{n+2}(R)$ that is the $R$-module of strictly upper-triangular $(n+2)$-dimensional square matrices. The group $\mr{U}_{n+2}(R)$ acts continuously on $\mathfrak{U}_{n+2}(R)$ by conjugation. We consider a $R\ps{\mr{U}_{n+2}(R)}$-submodule $\mathfrak{U}_{a,b}(R)$ of $\mathfrak{U}_{n+2}(R)$ given by
\[
	\{ x=(x_{ij}) \in M_{n+2}(R) \ | \ x_{ij} =0 \text{ if } j \le a+1 \text{ or } i \ge a+2 \}.
\]
In other words, breaking $M_{n+2}(R)$ into blocks using the partition $n+2=(a+1)+(b+1)$ and using block-matrix notation, we have
\[
\mathfrak{U}_{a,b}(R) = \ttmat{0}{M_{a+1,b+1}(R)}{0}{0}.
\]

Given a partial defining system $(\phi, \theta)$, we consider  $\mathfrak{U}_{a,b}(R)$ as a $G$-module via the continuous homomorphism
\[
G \to \mr{U}_{n+2}(R), \quad g \mapsto \ttmat{\phi(g)}{0}{0}{\theta(g)}.
\]
We define an $R\ps{G}$-module $\mf{U}_{\phi,\theta}(T)$ as $\mathfrak{U}_{a,b}(R) \otimes_R T$ with the diagonal $G$-action. 
We also have the following equivalent definition, which has the benefit of being more explicit:
\begin{itemize}
\item $\mf{U}_{\phi,\theta}(T) = M_{a+1,b+1}(T)$ as an $R$-module,
\item the action map $G \to \End(M_{a+1,b+1}(T))$ is given, for $g \in G$ and $x \in M_{a+1,b+1}(T)$, by
\[
g \star x = \phi(g) \cdot gx \cdot \theta(g)^{-1}.
\]
where $gx$ means apply the $g$ action on $T$ to each matrix entry, and the multiplication denoted by ``$\cdot$'' is of matrices.
\end{itemize}

Going forward, we use the latter description of $\mf{U}_{\phi,\theta}(T)$, so consider it as consisting of $(a+1)$-by-$(b+1)$ matrices,
rather than as a subgroup of $M_{n+2}(R)$.  Note that $\mf{U}_{\phi,\theta}(T)$ contains a copy of $T$ as an $R\ps{G}$-submodule by inclusion in the $(1,b+1)$-entry. Let 
$$
	\mf{U}'_{\phi,\theta}(T) = \mf{U}_{\phi,\theta}(T)/T.
$$ 
Let $x \mapsto \tilde{x}$ denote the $R$-module section $\mf{U}_{\phi,\theta}'(T) \to \mf{U}_{\phi,\theta}(T)$ given by filling in the $(1,b+1)$-entry as $0$.

\begin{lemma} \label{correspondence}
    Let $(\phi,\theta)$ be a partial defining system for Massey products $(\alpha,\sdot,\beta)$. 
    Then the map that takes a continuous $1$-cocycle $\kappa' \colon 
    G \to \mf{U}'_{\phi,\theta}(T)$ to a map $\rho \colon G \to \mc{U}'_{n+2}(T,a)$ given by
     \[
    	\rho=\ttmat{\phi}{\kappa'\theta}{0}{\theta}.
    \]
     is a bijection between $Z^1(G,\mf{U}'_{\phi,\theta}(T))$ and the set of proper defining systems in $T$ relative to $(\phi,\theta)$.
\end{lemma}

\begin{proof}
    Given a cochain $\kappa' \colon G \to \mf{U}'_{\phi,\theta}(T)$, set $\kappa = \kappa'\theta
    \colon G \to \mf{U}'_{\phi,\theta}(T)$.
    We have to check that 
    $$
    	\rho = \ttmat{\phi}{\kappa}{0}{\theta}
    $$ is a cocycle if and only if $\kappa'$ is a cocycle.   Matrix multiplication tells us that $\rho$ is a cocycle if and only if
    \begin{equation} \label{cocyclecond}
   	 \kappa(gh)= \phi(g) g\kappa(h)  + \kappa(g)\theta(h).
    \end{equation}
    The cochain $\kappa'$ is a cocycle if and only if the second equality holds in the following string of equalities
    \begin{align*}
        \kappa(gh)& =\kappa'(gh)\theta(gh) \\
        & =( g\star \kappa'(h)+\kappa'(g))\theta(gh) \\
        & =( \phi(g) g\kappa'(h) \theta(g)^{-1} + \kappa'(g)) \theta(g) \theta(h)\\
        &= \phi(g) g\kappa'(h) \theta(h) + \kappa'(g)\theta(g) \theta(h) \\
        &= \phi(g) g\kappa(h)  + \kappa(g)\theta(h),
    \end{align*}
    hence the result.
\end{proof}

The value of the Massey product associated to a proper defining system is also a value of a connecting homomorphism for
an exact sequence attached to the underlying partial defining system.

\begin{theorem} \label{connecting}
    Let $(\phi,\theta)$ be a partial defining system for $(\alpha,\sdot,\beta)$. 
    Let $\kappa' \in Z^1(G,\mf{U}'_{\phi,\theta}(T))$, and let $\rho = \sm{ \phi }{ \kappa' \theta}{}{\theta}$ be the associated proper
    defining system as 
    in Lemma \ref{correspondence}. 
    Consider the short exact sequence
    \[
    	0 \to T \to \mf{U}_{\phi,\theta}(T) \to \mf{U}'_{\phi,\theta}(T) \to 0.
    \]
    Then the image of the class of $\kappa'$ under the connecting map
    \[
    \partial \colon H^1(G,\mf{U}'_{\phi,\theta}(T)) \to H^2(G,T)
    \]
    is the $(n+1)$-fold Massey product $(\alpha_1,\dots,\alpha_a, \kappa'_{a+1,1}, \beta_1,\dots,\beta_b)_{\rho}$.
\end{theorem}

\begin{proof}
	Let $\kappa = \kappa' \theta \colon G \to \mf{U}'_{\phi,\theta}(T)$, and let $\tilde{\kappa}$ be its unique lift 
	to $\mf{U}_{\phi,\theta}(T)$ with $\tilde{\kappa}(g)$ having zero $(1,b+1)$-entry for all $g \in G$. The map
	$\tilde{\kappa}' = \tilde{\kappa}\theta^{-1} \colon G \to \mf{U}_{\phi,\theta}(T)$ is then a lift of $\kappa'$.
   	 By definition, the image of $\kappa$ is represented by the $2$-cocycle that is given by taking
    	the $(1,b+1)$-entry of $d\tilde{\kappa}'$. 
	We have
    	\begin{align*}
    		d\tilde{\kappa}'(g,h) &= \tilde{\kappa}'(g) + g \star \tilde{\kappa}'(h) - \tilde{\kappa}'(gh) \\
		&= \tilde{\kappa}(g)\theta(g)^{-1} + \phi(g) g \tilde{\kappa}(h) \theta(h)^{-1} \theta(g)^{-1} - \tilde{\kappa}(gh)\theta(gh)^{-1}\\
		&= (\tilde{\kappa}(g)\theta(h) + \phi(g)g\tilde{\kappa}(h) - \tilde{\kappa}(gh))\theta(gh)^{-1}.
    	\end{align*}
 	Since $\kappa$ satisfies \eqref{cocyclecond}, we have $\tilde{\kappa}(g)\theta(h) + \phi(g)g\tilde{\kappa}(h) - \tilde{\kappa}(gh) \in T$,
	and $T$ is fixed under the action of right multiplication by an element of $\mr{U}_{b+1}(R)$.  Since $\tilde{\kappa}(gh)$
	has zero $(1,b+1)$-entry, the $(1,b+1)$-entries of 
    	$d\tilde{\kappa}'(g,h)$ and $\tilde{\kappa}(g)\theta(h) + \phi(g)g\tilde{\kappa}(h)$ are equal.
    
    	The Massey product $(\alpha_1,\dots,\alpha_a, \kappa_{a+1,1}, \beta_1,\dots,\beta_b)_{\rho}$ (and note that
	$\kappa_{a+1,1} = \kappa'_{a+1,1}$) is the $(1,n+2)$-entry of 
    	$\tilde{\rho}(g) \cdot g\tilde{\rho}(h)$,  
	where
	$$
		\tilde{\rho} = \ttmat{\phi}{\tilde{\kappa}}{0}{\theta}.
	$$
	The result then follows from the fact that
	$$
		\tilde{\rho}(g) \cdot g\tilde{\rho}(h) = \ttmat{\phi(g)}{\tilde{\kappa}(g)}{0}{\theta(g)}\ttmat{\phi(h)}{g\tilde{\kappa}(h)}{0}{\theta(h)}
		= \ttmat{\phi(gh)}{\phi(g)g\tilde{\kappa}(h) + \tilde{\kappa}(g)\theta(h)}{0}{\theta(gh)}.
	$$
\end{proof}

In fact, the proof of Theorem \ref{connecting} gives an explicit map $Z^1(G,\mf{U}'_{\phi,\theta}(T)) \to Z^2(G,T)$, taking
a $1$-cocycle $\kappa'$ to the $(1,b+1)$-entry of $d\tilde{\kappa}'$, for the specific lift $\tilde{\kappa}'$ of $\kappa'$ defined therein.

\section{Massey products as values of Bockstein maps} \label{massey}
We return to the setting and notation of Section \ref{sec:bock}.
We first discuss a general result that gives partial information about the generalized Bockstein map $\Psi^{(n)}$ in terms of Massey products. Then we discuss specific examples where this information completely determines $\Psi^{(n)}$.

\subsection{Partial defining systems and Bockstein maps}
\label{subsec:general partial defining systems}

Fix integers $a,b \ge 0$ such that $a+b=n$ and group homomorphisms
\begin{eqnarray*}
	\phi \colon H \to \mr{U}_{a+1}(R) &\mr{and}& \theta \colon H \to \mr{U}_{b+1}(R),
\end{eqnarray*}
so viewing $\phi$ and $\theta$ as maps from $G$ via precomposition with the quotient map, 
the pair $(\phi,\theta)$ is an $(a,b)$-partial defining system.  
We let $\alpha = (\phi_{i,i+1})_i$ and $\beta = (\theta_{i,i+1})_i$, so this partial defining system is of
Massey products $(\alpha, \sdot, \beta)$.
If $b = 0$, we often refer to the pair $(\phi,\theta)$ simply as $\phi$.

\begin{lemma} \label{mapnilpotent}
    Let $e \in \mathfrak{U}_{a,b}(R)$ be the matrix with $(a+1,1)$-entry equal to $1$ and all other entries $0$.
    There is a continuous $R\ps{G}$-module homomorphism $p_{\phi,\theta} \colon \Omega/I^{n+1} \to \mathfrak{U}_{a,b}(R)$ given on
    the cosets of images of group elements by 
    $$
    	p_{\phi,\theta}([h]) = \phi(h) \cdot e \cdot \theta(h)^{-1}.
    $$ 
    The image of $I^n$ is 
    contained in the submodule of matrices that are zero outside of their $(1,b+1)$-entries.
\end{lemma}

\begin{proof}
The map $\tilde{p}_{\phi,\theta} \colon \Omega \to \mathfrak{U}_{a,b}(R)$ inducing $p_{\phi,\theta}$ is given by the action of $H$ on $e \in \mathfrak{U}_{a,b}(R)$ via the composite homomorphism
\[
H \xrightarrow{\rho_{\phi,\theta}} \mr{U}_{n+2}(R) \xrightarrow{\mathrm{ad}} \Aut(\mathfrak{U}_{a,b}(R)),
\]
where $\rho_{\phi,\theta} \colon  H \to \mr{U}_{n+2}(R)$ is given by
$$
	\rho_{\phi,\theta}(h) = \ttmat{\phi(h)}{0}{0}{\theta(h)}
$$ 
and $\mathrm{ad}$ denotes the conjugation action. The action of $G$ on $\Omega$ is given by the homomorphism $G \to H$, and the action of $G$ on $\mathfrak{U}_{a,b}(R)$ is given by the composite of this map with $H \to \Aut(\mathfrak{U}_{a,b}(R))$, so $\tilde{p}_{\phi,\theta}$ is $G$-equivariant.  We must show it factors through $\Omega/I^{n+1}$.

Let $J \subset R\ps{\mr{U}_{n+2}(R)}$ be the augmentation ideal.  Since the $H$-action factors through $\mr{U}_{n+2}(R)$, we have $I^k\mathfrak{U}_{a,b}(R) \subseteq J^k\mf{U}_{a,b}(R)$ for all $k$.  It is easy to see inductively that
\[
J^k\mathfrak{U}_{n+2}(R) = \{ (a_{ij}) \in M_{n+2}(R) \ | \ a_{ij}=0 \text{ if } j-i \le k\}.
\]
In particular, $J^{n+1}\mathfrak{U}_{n+2}(R)=0$ and 
\[
J^n\mathfrak{U}_{n+2}(R)= \{ (a_{ij}) \in M_{n+2}(R) \ | \ a_{ij}=0 \text{ if } (i,j) \ne (1,n+2)\}.
\]
Still viewing $\mf{U}_{a,b}(R)$ as a subgroup of $\mf{U}_{n+2}(R)$, the containments 
\[
I^k\mathfrak{U}_{a,b}(R) \subseteq J^k\mathfrak{U}_{a,b}(R) \subseteq J^k\mathfrak{U}_{n+2}(R)
\]
imply the result.
\end{proof}

Lemma \ref{mapnilpotent} implies that there is a map of short exact sequences of $R\ps{G}$-modules
\begin{equation}
\label{eq:general diagram}
\begin{tikzcd}
0 \arrow[r] & T \otimes_R I^n/I^{n+1} \arrow[r] \arrow{d}{p_{\phi,\theta}} & T \otimes_R \Omega/I^{n+1} \arrow[r] \arrow{d}{p_{\phi,\theta}} & T \otimes_R \Omega/I^n \arrow[r] \arrow{d}{p_{\phi,\theta}} & 0 \\
0 \arrow[r] & T \arrow[r] & \mf{U}_{\phi,\theta}(T) \arrow[r] & \mf{U}'_{\phi,\theta}(T) \arrow[r] & 0,
\end{tikzcd}
\end{equation}
where $p_{\phi,\theta}$ is the tensor product with $T$ of the map in Lemma \ref{mapnilpotent} coming from $\rho_{\phi,\theta}$.  As a direct consequence of this commutativity and Theorem \ref{connecting}, we have the following.

\begin{theorem}
\label{prop:general}
Let $\phi \colon G \to \mr{U}_{a+1}(R)$ and $\theta \colon G \to \mr{U}_{b+1}(R)$ restrict to $\alpha \in Z^1(G,R)^a$ and $\beta \in 
Z^1(G,R)^b$ as above.
Let $f \in Z^1(G,T \otimes_R \Omega/I^n)$, and let $\rho$ denote the proper defining system relative to $(\phi,\theta)$ 
associated to $p_{\phi,\theta} \circ f$ by Lemma \ref{correspondence}.  Then we have
\[
p_{\phi,\theta} (\Psi^{(n)}([f])) = (\alpha,(p_{\phi,\theta} \circ f)_{a+1,1},\beta)_{\rho}
\]
in $H^2(G,T)$.  Here, the maps $p_{\phi,\theta}$ on the left and right are those induced on cohomology by
the left and right vertical maps in \eqref{eq:general diagram}.
\end{theorem}

We will give examples of groups $H$ and integers $n$ such that there is a set $X$ of choices of $(\phi,\theta)$ for which the map
\[
	H^2(G,T) \otimes_R I^n/I^{n+1} \xrightarrow{\prod_{(\phi,\theta) \in X} p_{\phi,\theta}} \prod_{(\phi, \theta) \in X} H^2(G,T)
\]
is injective. In such cases, Theorem \ref{prop:general} shows that the generalized Bockstein map $\Psi^{(n)}$ is determined by Massey products. In the rest of this section, we consider some specific examples in detail.

\subsection{Unipotent binomial matrices}
\label{subsec:unipotent binomials}

We introduce the \emph{unipotent binomial matrices}, which are a source of many partial defining systems. Let $n$ denote a positive integer, and let $p$ be a prime number.  

Let $u_n$ denote the $(n+1)$-dimensional nilpotent upper triangular matrix 
\[
	u_n = \begin{pmatrix}
		0 & 1 & 0 & \cdots & 0 \\
		& 0 & 1 & \ddots & \vdots  \\
		&& 0 & \ddots & 0 \\
		&&& \ddots & 1  \\
		&&&& 0 
	\end{pmatrix}.
\]
For any $k \ge 1$, the matrix $u_n^k$ has $(i,j)$-entry $1$ if $j-i=k$ and $0$ otherwise.  In particular, we have $u_n^{n+1} = 0$.
 
Let $\binmat{\sdot}{n} \colon \zp \to \mr{U}_{n+1}(\zp)$ denote the unique continuous homomorphism to $(n+1)$-dimensional unipotent matrices with $\zp$-entries such that $\binmat{1}{n}=1+u_n$. By the binomial theorem, for $a\in \Z$ we have
\[
	\binmat{a}{n} = (1+u_n)^a = \sum_{k=0}^n \binom{a}{k} u_n^k
	= \begin{pmatrix}
		1 & a & \binom{a}{2} & \cdots & \binom{a}{n} \\
		& 1 & a & \ddots & \vdots  \\
		&& 1 & \ddots & \binom{a}{2} \\
		&&& \ddots & a  \\
		&&&& 1 
	\end{pmatrix}.
\]
If $t \ge s$ and $n < p^{t-s+1}$, then the composite map
\[
\Z \xrightarrow{\binmat{\sdot}{n}} \mr{U}_{n+1}(\Z_p) \to \mr{U}_{n+1}(\Z/p^s\Z)
\]
that sends $a$ to $(1+u_n)^a$ modulo $p^s$ factors through $\Z/p^t\Z$ by Lemma \ref{lem:binomial} applied with $x = u_n$.  By abuse of notation, we again denote the resulting map $\Z/p^t\Z \to  \mr{U}_{n+1}(\Z/p^s\Z)$ by $\binmat{\sdot}{n}$. 
In particular, the map $\binom{\sdot}{n} \colon \Z \to \Z/p^s\Z$ given by $a \mapsto \binom{a}{n} \pmod{p^s}$  factors through $\Z/p^t\Z$, and we abuse notation to also denote the resulting map $\Z/p^t\Z \to \Z/p^s\Z$ by $\binom{\sdot}{n}$.

The following lemma, phrased conveniently for our purposes, summarizes the above discussion.

\begin{lemma} \label{binommap}
	Let $A$ be a quotient of the ring $\zp$ and $R$ be a quotient of $A$.  Let $H$ be a profinite group and 
	$\chi \colon H \to A$ be a continuous homomorphism.  
	Suppose that either $A = \zp$ or $|R| < \frac{p}{n} |A|$.  Then there is a homomorphism
	$$
		\binmat{\chi}{n} \colon H \to \mr{U}_{n+1}(R),
	$$
	defined by $\binmat{\chi}{n}(h) = \binmat{\chi(h)}{n}$ for all $h\in H$.
\end{lemma}

\begin{proof}
	If $|R| = p^s$ and $|A| = p^t$, then $|R| < \frac{p}{n}|A|$ if and only if $n < p^{t-s+1}$.
\end{proof}

\subsection{Procyclic case} \label{procyclic}

In this subsection, we fix a surjective homomorphism $\chi \colon G \to A$, where $A$ is a nonzero quotient of $\Z_p$.  We suppose that our ring $R$
is a nonzero quotient of $A$ with $A = \zp$ or $n|R| < p|A|$. We define $H$ to be the coimage of $\chi$, so $H \cong A$. We fix $h \in H$ to be the preimage of $1 \in A$ and let $x=[h]-1 \in \Omega$, which is a generator of the augmentation ideal $I$. Our assumption on the size of $R$ implies that $\Omega/I^j \cong R[x]/(x^j)$ for all $j \le n+1$ by the discussion of Section \ref{abelian}. In particular, we have $I^n/I^{n+1} = R x^n$.

The $(n,0)$-proper defining systems relative to $\phi = \binmat{\chi}{n}$ and the trivial map $\theta$ to $\mr{U}_1(R) = \{1\}$ agree with the proper defining systems considered in \cite{sh-massey} for Galois groups.  We give an interpretation of the resulting Massey products in terms of generalized Bockstein maps.  That is, let us apply the discussion of Section \ref{subsec:general partial defining systems} to this situation.  We have $\alpha=(\chi,\dots,\chi) \in Z^1(G,R)^n$, which we denote by $\chi^{(n)}$. We denote $\mf{U}_{\phi,\theta}(T)$ by $\mf{U}_{\binmat{\chi}{n}}(T)$.
 
Set $p_n = p_{\binmat{\chi}{n},0}$ for brevity.  Then the diagram \eqref{eq:general diagram} becomes
\[
\begin{tikzcd}
0 \arrow[r] & T \cdot x^n \arrow[r] \arrow{d} & T \otimes_R \Omega/I^{n+1} \arrow[r] \arrow{d}{p_n} & T \otimes_R \Omega/I^n \arrow[r] \arrow{d}{p_n} & 0 \\
0 \arrow[r] & T \arrow[r] & \mf{U}_{\binmat{\chi}{n}}(T) \arrow[r] & \mf{U}'_{\binmat{\chi}{n}}(T) \arrow[r] & 0,
\end{tikzcd}
\]
where $p_n$ is the map attached to $(\binmat{\chi}{n},0)$ by Lemma \ref{mapnilpotent}.
Explicitly, the vector $p_n(\sum_{k=0}^n a_k x^k)$ in $M_{n+1,1}(T)$ has $i$th entry $a_{n+1-i}$.  (See
the more general case proven in Lemma \ref{bicycform} of the next subsection.)

By Lemma \ref{correspondence}, it follows that the proper defining system $\rho_{x^n}$ relative to $\binmat{\chi}{n}$ that is attached to $p_n \circ f$, where
$$
	f = \sum_{k=0}^{n-1} \lambda_k x^k \in Z^1(G,T \otimes_R \Omega/I^n),
$$ 
satisfies $(\rho_{x^n})_{n+1-k,n+2} = \lambda_k$ for $0 \le k \le n-1$.   In particular, the element $\lambda = \lambda_0 = (\rho_{x^n})_{n+1,n+2}$ is the image of $f$ under the map 
$$
	Z^1(G,T \otimes_R \Omega/I^n) \to Z^1(G,T)
$$ 
induced by the augmentation $\Omega/I^n \mapsto \Omega/I = R$.  

Theorem \ref{prop:general} then gives us an explicit description of the values of the generalized Bockstein homomorphism on classes in $H^1(G,T \otimes_R \Omega/I^n)$ as Massey products $(\chi^{(n)},\sdot)$ relative to $\binmat{\chi}{n}$, as follows.

\begin{theorem} \label{cyclic_case}
	For $f \in Z^1(G,T \otimes_R \Omega/I^n)$, we have
	$$
		\Psi^{(n)}([f]) = (\chi^{(n)},\lambda)_{\rho_{x^n}} \cdot x^n,
	$$
	where $\rho_{x^n}$ is the proper defining system relative to $\binmat{\chi}{n}$ attached to $f$, and
	$\lambda$ is the image of $f$ in $Z^1(G,T)$.
\end{theorem}

In particular, we have the following description of the image of $\Psi^{(n)}$.

\begin{corollary} \label{image_cyclic_case}
	The image of the generalized Bockstein map $\Psi^{(n)}$ is the set of all $(\chi^{(n)},\lambda)_{\rho} \cdot x^n$ for 
	Massey products of $n$ copies of $\chi$ with $1$-cocycles $\lambda \in Z^1(G,T)$ for proper defining systems $\rho$
	relative to $\binmat{\chi}{n}$ with $\rho_{n+1,n+2} = \lambda$.
\end{corollary}

Theorem \ref{augfilt} provides the following application to the graded quotients of Iwasawa cohomology groups of $N = \ker (\chi
\colon G \to H)$.

\begin{corollary}
	Suppose that $G$ is $p$-cohomologically finite of $p$-cohomological dimension $2$.
	Let $P_n(H)$ denote the subgroup of $H^2(G,T) \otimes_R I^n/I^{n+1}$ generated by all 
	$(\chi^{(n)},\lambda)_{\rho} \cdot x^n$ for proper defining systems $\rho$ relative to
	$\binmat{\chi}{n}$ and $\lambda = \rho_{n+1,n+2}$.
	We have a canonical isomorphism of $R$-modules
	$$
		\frac{I^nH^2_{\Iw}(N,T)}{I^{n+1}H^2_{\Iw}(N,T)} \cong \frac{H^2(G,T) \otimes_R I^n/I^{n+1}}{P_n(H)}.
	$$
\end{corollary}

\subsection{Pro-bicyclic case} \label{bicyc_case}

In this subsection, we
\begin{itemize}
\item fix a surjective homomorphism $(\chi,\psi) \colon G \twoheadrightarrow A \times B$, where $A$ and $B$ are nonzero quotients of $\Z_p$,
\item let $a,b \ge 0$ denote integers such that $a+b=n$, and
\item suppose that 
$R$ is a nonzero quotient of both $A$ and $B$ with $a|R| < p|A|$ if $A$ is finite and $b|R| < p|B|$ if $B$ is finite.
\end{itemize}
Let $H$ be the coimage of $(\chi,\psi)$ so that $H \cong A \times B$.   
Let $h_A,h_B \in H$ be the preimages of $(1,0),(0,1) \in A \times B$, respectively, and let $x = [h_A]-1$ and $y=[h_B]-1$
so that $(x,y)$ is the augmentation ideal $I$ of $\Omega = R\ps{H}$. We have $\Omega/I^j = R[x,y]/(x,y)^j$ for all $j \le n$. In particular, we have
\[
I^n/I^{n+1} = \bigoplus_{i+j=n} R x^iy^j.
\]

We apply the discussion of Section \ref{subsec:general partial defining systems} to this situation. We take $\phi = \binmat{\chi}{a} \colon H \to \mr{U}_a(R)$ and $\theta = \binmat{\psi}{b} \colon H \to \mr{U}_b(R)$.  We have $\alpha=\chi^{(a)} \in Z^1(G,R)^a$ and $\beta=\psi^{(b)} \in Z^1(G,R)^b$.

Set $p_{a,b} = p_{\binmat{\chi}{a},\binmat{\psi}{b}}$ for brevity.
In this setting, the diagram \eqref{eq:general diagram} becomes
\begin{equation} \label{bicyc_diag}
\begin{tikzcd}
0 \arrow[r] &  \bigoplus_{i+j=n} T \cdot x^iy^j  \arrow[r] \arrow{d} & T \otimes_R \Omega/I^{n+1} \arrow[r] \arrow{d}{p_{a,b}} & T \otimes_R \Omega/I^n \arrow[r] \arrow{d}{p_{a,b}} & 0 \\
0 \arrow[r] & T \arrow[r] & \mf{U}_{\binmat{\chi}{a},\binmat{\psi}{b}}(T) \arrow[r] & \mf{U}'_{\binmat{\chi}{a},\binmat{\psi}{b}}(T) \arrow[r] & 0.
\end{tikzcd}
\end{equation}

\begin{lemma} \label{bicycform}
	The $R\ps{G}$-module map $p_{a,b} \colon T \otimes_R \Omega/I^{n+1} \to \mf{U}_{\binmat{\chi}{a},\binmat{\psi}{b}}(T)$ 
	is an isomorphism satisfying
	$$
		p_{a,b}\left(\sum_{k_1+k_2 \le n} c_{k_1,k_2} x^{k_1}y^{k_2}\right)
		= (c_{a+1-i,j-1})_{i,j}.
	$$
	In particular, the left-hand vertical map in \eqref{bicyc_diag} is 
	given by projection onto the factor $T \cdot x^a y^b \cong T$.
\end{lemma}

\begin{proof} 
	This reduces immediately to the case that $T = R$, since we can obtain the case of arbitrary $T$ by $R$-tensor product with
	the identity of $T$.  
	Let $e$ be as in Lemma \ref{mapnilpotent}, the matrix with a single nonzero entry of $1$ in the $(a+1,1)$-coordinate of
	$M_{a+1,b+1}(R)$.  The $(i,j)$-entry of $g \star e = \binmat{\chi(g)}{a} e \binmat{\psi(g)}{b}$ is
	\[
		\binom{\chi(g)}{a+1-i}\binom{\psi(g)}{j-1}.
	\]
	which agrees with the coefficient of $x^{a+1-i}y^{j-1}$ in $g \cdot 1$ by \eqref{act_sum}.
\end{proof}

\begin{corollary}
\label{cor:bicyclic}
	For
	$$
		f = \sum_{k_1+k_2 < n} \lambda_{k_1,k_2} x^{k_1}y^{k_2} \in Z^1(G,\Omega/I^n \otimes_R T)
	$$ 
	and $\rho_{x^ay^b}$ the proper defining system relative to $(\binmat{\chi}{a},\binmat{\psi}{b})$ attached to $p_{a,b} \circ f$
	by Lemma \ref{correspondence}, we have
	$$
		(\rho_{x^ay^b})_{a+1-k_1,a+2+k_2} = \lambda_{k_1,k_2}
	$$
	for all $0 \le (k_1,k_2) < (a,b)$.  In particular, we have $(\rho_{x^ay^b})_{a+1,a+2} = \lambda_{0,0}$, which is the
	image of $f$ in $Z^1(G,T)$ under the map induced by the quotient $\Omega/I^n \to \Omega/I = R$.
\end{corollary}

The following is then a direct consequence of Theorem \ref{prop:general}.

\begin{theorem} \label{bicyc_image}
	For $f \in Z^1(G,\Omega/I^n \otimes_R T)$, the image of $\Psi^{(n)}([f])$ in
	$$
		H^2(G,T) \otimes_R I^n/I^{n+1} \cong \bigoplus_{a+b=n} 
		H^2(G, T) \cdot x^ay^b
	$$ 
	is
	$$
		\sum_{a+b = n} (\chi^{(a)},\lambda,\psi^{(b)})_{\rho_{x^ay^b}} \cdot x^a y^b,
	$$ 
	where $\rho_{x^ay^b}$ is the proper defining system relative to $(\binmat{\chi}{a},\binmat{\psi}{b})$ attached to 
	$p_{a,b} \circ f$, and $\lambda$ is the image of $f$ in $Z^1(G,T)$.
\end{theorem}

Applying Theorem \ref{augfilt}, we obtain the following description of graded quotients of Iwasawa cohomology.

\begin{corollary}\label{cor:bicyc}
	Suppose that $G$ is $p$-cohomologically finite of $p$-cohomological dimension $2$.
	Let $P_n(H)$ denote the subgroup of $H^2(G,T) \otimes_R I^n/I^{n+1}$ consisting of all sums
	$\sum_{a+b = n} (\chi^{(a)},\lambda,\psi^{(b)})_{\rho_{x^ay^b}} \cdot x^a y^b$, where the $\rho_{x^ay^b}$ and $\lambda$ 
	are associated to a cocycle in $Z^1(G,\Omega/I^n \otimes_R T)$
	as in Proposition \ref{bicyc_image}.  
	We then have a canonical isomorphism of $R$-modules
	$$
	\frac{I^nH^2_{\Iw}(N,T)}{I^{n+1}H^2_{\Iw}(N,T)} \cong \frac{H^2(G,T) \otimes_R I^n/I^{n+1}}{P_n(H)}.
	$$
\end{corollary}

\subsection{Elementary abelian $p$-groups}
\label{elem}
The pattern seen in the cyclic and bicyclic cases does not continue for all finitely generated abelian pro-$p$ groups. 
To see why, consider the case that $H \cong \F_p^3$ with basis $(\gamma_1,\gamma_2,\gamma_3)$
and dual basis $(\chi_1,\chi_2,\chi_3)$.  For $x_i = [\gamma_i]-1$, we have a basis of $I^n/I^{n+1}$ consisting of monomials 
$x_1^{i_1}x_2^{i_2}x_3^{i_3}$ with $i_1+i_2+i_3=n$. 
Following the pattern of the cyclic and bicyclic cases, one might guess that the coefficient of $x_1^{i_1}x_2^{i_2}x_3^{i_3}$ in 
$\Psi^{(n)}([f])$ is an $(n+1)$-fold Massey product involving $i_j$ copies of each $\chi_j$ 
and another cocycle $\lambda$ determined by $f$.  However, this pattern fails already for $n=3$ and the coefficient of $x_1x_2x_3$: 
any 4-fold Massey product involving the $\chi_i$ must have two of these characters beside each other, and thus, to be defined, the 
cup product of those two characters must vanish. Since these cup products will not vanish in general, we cannot hope for such a general statement to hold.

Nevertheless, at least in some cases, one can still describe the generalized Bockstein maps $\Psi^{(n)}$ in terms of Massey products, at the expense of taking a non-standard basis for $I^n/I^{n+1}$.  In this subsection, we assume that $H \cong \F_p^r$ for some $r \ge 1$.  Correspondingly, we take $n < p$ and $R = \F_p$.

We let $V^\vee=\Hom(V,\F_p)$ for an abelian group $V$. For any element $\chi \in H^\vee$, we have a homomorphism
\[
	\binmat{\chi}{n} \colon H \to \mathrm{U}_{n+1}(\F_p).
\]
Precomposing with $G \to H$, we may view $\chi$ as a character of $G$.
This gives an $(n,0)$-partial defining system, and we set $p_{\chi,n} = p_{\binmat{\chi}{n},0}$ for brevity. By \eqref{eq:general diagram}, the map $p_{\chi,n}$ induces a map $p_{\chi,n} \colon I^n/I^{n+1} \to \F_p$, so $p_{\chi,n} \in (I^n/I^{n+1})^\vee$. This defines a function $p_{-,n} \colon H^\vee \to (I^n/I^{n+1})^\vee$. 

Let us fix an isomorphism $H \xrightarrow{\sim} \F_p^r$, which in turn fixes an ordered dual basis
$(\gamma_i)_{i=1}^r$ of $H$.  Setting $x_i = [\gamma_i] - 1 \in \Omega$, this  provides an identification 
\begin{equation} \label{ident}
    	\Omega/I^{n+1} = \F_p[x_1,\dots,x_n]/(x_1, \ldots, x_r)^{n+1}.
\end{equation}
Then $I^n/I^{n+1}$ has a basis given by 
$x_1^{d_1}\cdots x_r^{d_r}$ with $(d_1,\ldots,d_r)$ ranging over $r$-tuples of nonnegative integers with $d_1+\dots +d_r=n$.  
We compute $p_{\chi,n}$ on this basis. 

\begin{lemma} \label{projmap}
	Let $\chi \in H^{\vee}$.  For any nonnegative integers $d_1, \ldots, d_r$ with sum $n$, we have
	\[
    		p_{\chi,n}(x_1^{d_1}\cdots x_r^{d_r})=\prod_{i=1}^r \chi(\gamma_i)^{d_i}.
	\]
\end{lemma}

\begin{proof}
	We have
	$$
		p_{\chi,n}(x_i) = \left(\binmat{\chi(\gamma_i)}{n}-1 \right)e=((1+u_n)^{\chi(\gamma_i)}-1)e \in \mathfrak{U}_{\binmat{\chi}{n}},
	$$ 
	where $u_n$ is as in Section \ref{subsec:unipotent binomials} and $e$ is as in Lemma \ref{mapnilpotent}. 
	Note that $u_n^n$ has a $1$ in its $(1,n+1)$ entry and all other entries $0$, and $u_n^{n+1} = 0$.
	For $d_1+\dots+d_r=n$, the value $p_{\chi,n}(x_1^{d_1}\cdots x_r^{d_r})$ is the $(1,n+1)$-entry of the matrix
    	$$
    		\prod_{i=1}^r ((1+u_n)^{\chi(\gamma_i)}-1)^{d_i} = \prod_{i=1}^r (\chi(\gamma_i)u_n)^{d_i} =  
		\left(\prod_{i=1}^r \chi(\gamma_i)^{d_i} \right)  u_n^n,
	$$
	proving the lemma.
\end{proof}	
	
The key lemma is then the following.

\begin{lemma}
\label{lem:lin alg}
The image of $p_{-,n} \colon H^\vee \to (I^n/I^{n+1})^\vee$ generates $(I^n/I^{n+1})^\vee$.
\end{lemma}

\begin{proof}   
    	Using our identification \eqref{ident}, any non-zero $F \in I^n/I^{n+1}$ has a unique representative also denoted $F$
	in $\F_p[x_1,\dots,x_r]$ that is homogeneous of degree $n$ and Lemma \ref{projmap} implies that
    	\[
    		p_{\chi,n}(F) = F(\chi(\gamma_1),\dots,\chi(\gamma_r)).
    	\]
    	Writing $\Gamma \colon H^\vee \to \F_p^r$ for the isomorphism given by $\Gamma(\chi) = (\chi(\gamma_1),\dots,\chi(\gamma_r))$, this can 
	be succinctly written as $p_{\chi,n}(F)=F(\Gamma(\chi))$.
    
    	For a finite set $S$ and an $s \in S$, we denote by $\F_p^S$ the vector space of functions $S \to \F_p$ and by 
	$1_s \in \F_p^S$ the indicator function of $s$.  
    	The lemma may then be rephrased as the statement that the linearization $\tilde{p}_{-,n}$ of $p_{-,n}$, given by
    	\[
    		\tilde{p}_{-,n} \colon \F_p^{H^\vee} \to (I^n/I^{n+1})^\vee,  \quad 1_\chi \mapsto p_{\chi,n}
    	\]
    	is surjective, or equivalently that the dual map
    	\[
    		\tilde{p}_{-,n}^\vee \colon I^n/I^{n+1} \to (\F_p^{H^\vee})^\vee 
    	\]
    	is injective. For any non-zero $F \in I^n/I^{n+1}$,
	since $n<p$, the finite field Nullstellensatz provides the existence of $v \in \F_p^r$ for which $F(v) \ne 0$.
	Then 
	$$
		\tilde{p}_{-,n}^\vee(F)(1_{\Gamma^{-1}(v)}) = p_{\Gamma^{-1}(v),n}(F)= F(v) \ne 0,
	$$ 
	so $\tilde{p}_{-,n}^\vee(F) \ne 0$.
\end{proof}

\begin{remark}
\label{rem:n<p}
This lemma is the reason for our assumption that $n<p$ in this section rather than the assumption $n|R|<p|A|$ used in other sections. To see that this argument cannot work for torsion-free abelian groups $H$ and arbitrary $n$, take $R=\F_p$ and $H\cong \Z_p^2$. Then we know that $I^n/I^{n+1}$ has dimension $n+1$ for any $n$, and the proof of Lemma \ref{lem:lin alg} shows that $p_{-,n} \colon \Hom(H,\F_p) \to (I^n/I^{n+1})^\vee$ is homogeneous of degree $n$ in the sense that $p_{a\varphi,n} = a^np_{\varphi,n}$ for $\varphi \in \Hom(H,\F_p)$ and
$a \in \F_p$, so the span of its image has dimension at most the cardinality of $\Hom(H,\F_p)/\F_p^\times$, which is $p+1$.
\end{remark}

We now come to our result expressing values of the generalized Bockstein maps as sums of ``cyclic''
Massey products.  If $\chi \in H^{\vee}$ and $f \in Z^1(G, T \otimes_R \Omega/I^n)$, then we say that a proper defining system relative to
$\binmat{\chi}{n}$ is attached to $f$ if it is attached to the image of $f$ in $Z^1(G, T \otimes_R \Omega_{\chi}/I_{\chi}^n)$,
where $\Omega_{\chi} = R[H/\ker \chi]$ and $I_{\chi}$ is its augmentation ideal.

\begin{theorem}
There exist $N \ge 1$ and $\chi_1, \ldots, \chi_N \in H^\vee$ such that $(p_{\chi_i,n})_{i=1}^N$ is an ordered $\F_p$-basis of 
$(I^n/I^{n+1})^\vee$. For any such $(\chi_i)_{i=1}^N$, let $(y_i)_{i=1}^N$ be the basis of $I^n/I^{n+1}$ dual to $(p_{\chi_i,n})_{i=1}^N$.  Then for any $f \in Z^1(G, T \otimes_R \Omega/I^{n})$, we have
\[
\Psi^{(n)}([f]) = \sum_{i=1}^N  (\chi_i^{(n)},\lambda)_{\rho_i} \cdot y_i,
\]
where $\rho_i$ is the proper defining system relative to $\binmat{\chi_i}{n}$ attached to $f$ and $\lambda$ is the image of $f$ in $Z^1(G,T)$.
\end{theorem}

\begin{proof}
	The first statement is clear from Lemma \ref{lem:lin alg}. For the second statement, let
	\[
		\Psi^{(n)}([f]) = \sum_{i=1}^N  c_i \cdot y_i.
	\]
	for some $c_i \in H^2(G,T)$. Since $p_{\chi_1,n},\dots,p_{\chi_N,n}$ is the dual basis to $y_1,\dots,y_N$, we have 
	$c_i = p_{\chi_i,n}(\Psi^{(n)}([f]))$ for $1 \le i \le N$. But by Theorem \ref{prop:general}, we have 
	$$
		p_{\chi_i,n}(\Psi^{(n)}([f])) =  (\chi_i^{(n)},\lambda)_{\rho_i}.
	$$
\end{proof}

\subsection{Heisenberg case} \label{heis_case}
In this section, assume that $H=\mr{U}_3(A)$ for a nonzero quotient $A$ of $\Z_p$, and that $R$ is a quotient of $A$ such that either $R=\Z_p$ or $n|R| < p|A|$.  We study the generalized Bockstein maps $\Psi^{(n)}$ in the cases $n = 2$ and $n = 3$.

Let 
\begin{equation} \label{generators}
x= \left[\smthree{1}{0}{0}\right]-1, \quad y=\left[\smthree{0}{0}{1}\right]-1, \quad z= \left[\smthree{0}{1}{0}\right]-1 \in \Omega.
\end{equation}
Then $I$ is the two-sided ideal generated by $x$ and $y$, and $I/I^2 \cong Rx \oplus Ry$.  
Let $\chi, \psi \colon G \to A$ be the unique characters factoring
through $H$ and such that 
\[
	\chi\left(\smthree{1}{0}{0}\right) = 1, \quad  \chi\left(\smthree{0}{0}{1}\right) = 0, \quad \psi\left(\smthree{1}{0}{0}\right) = 0, 
	\quad \mr{and} \quad \psi\left(\smthree{0}{0}{1}\right) = 1.
\]
Then $(\chi, \psi) \colon G \to A \times A$ defines a homomorphism.

\begin{lemma} \label{gradquot}
	The $R$-module $I^2/I^3$ is freely generated by the image of the set 
	$$
		S_2 = \{ x^2, y^2, yx, z \},
	$$ 
	and $I^3/I^4$ is $R$-freely generated by the image of
	$$
		S_3=\{x^3,xz,yx^2,y^2x,y^3,yz\}.
	$$
\end{lemma}

\begin{proof}
	For any $n$, Lemma \ref{lem:binomial} and the condition that $n|R| < p|A|$ in the case that $A$ is finite are 
	enough to guarantee that the
	quotient $\Omega/I^{n+1}$ is isomorphic to the analogous quotient with $A$ replaced by $\zp$, so we may
	suppose in this proof that $H = \mr{U}_3(\zp)$.  
 
 	Let $\Sigma$ be the noncommutative $R\ps{z}$-power series ring $\Sigma$ in variables $x$ and $y$.
	It follows from the standard presentation of $\mr{U}_3(\zp)$ as a finitely generated pro-$p$ group that
	$\Omega = R\ps{\mr{U}_3(\zp)}$ is the quotient of $\Sigma$ by the ideal generated by
	\begin{equation} \label{w}
		w = (1+y)(1+x)z - (xy-yx).
	\end{equation}
	The augmentation ideal $I$ of $\Omega$ is $(x,y)$, so $I^n$ is generated by the monomials
	in $x$ and $y$ of degree at least $n$.  Using \eqref{w}, we can reduce this to
	$$
		I^n = (y^j x^i z^k \mid i+j+2k \ge n).
	$$
	It is therefore enough to check that the image of the set $S_n$ is $R$-linearly independent in $I^n/I^{n+1}$ for $n\in \{2,3\}$.
	
	Consider $\Sigma$ as a graded $R$-algebra with $x$, $y$, and $z$ in degrees $1$, $1$, and $2$, respectively.  Let $J_n$ denote the ideal of elements of $\Sigma$ of degree at least $n$.
	Suppose that $f\in \Sigma$ lies in the intersection of the $R$-span of the elements of $S_n$ with $(w)+J_{n+1}$. When $n=2$, one can easily see that $f=0$. When $n=3$, there are $a, b, c, d \in R$ such that 
	\[
		f+J_4 = (ax+by)w + w(cx+dy)+J_4.
	\]
	By the hypothesis on $f$, the degree $3$ terms above are in the $R$-span of $S_3$, which forces $a=b=c=d=0$, 
	and hence $f=0$.
\end{proof}

Let us first consider the case $n = 2$.  By Lemma \ref{gradquot}, we see that $I^2/I^3$ is a free
$R$-module on the set $S_2 = \{x^2,y^2,yx,z\}$.
We consider the three partial defining systems
\[
\phi_{x^2} = \binmat{\chi}{2}, \quad \phi_{y^2} = \binmat{\psi}{2}, \quad \phi_z \colon H \to \mr{U}_3(R) \times \mr{U}_1(R) = \mr{U}_3(R),
\]
with $a=2$ and $b=0$, where $\phi_z$ is the quotient map on coefficients, and the partial defining system
\[
\phi_{yx} = (\chi,\psi) \colon H \to \mr{U}_2(R) \times \mr{U}_2(R) = R \times R
\]
for $a=b=1$.  By Theorem \ref{connecting}, the partial defining systems $\phi_{x^2}$, $\phi_{y^2}$, $\phi_{z}$, and $\phi_{yx}$ correspond to Massey products $(\chi,\chi, \sdot)$, $(\psi,\psi,\sdot)$, $(\chi,\psi,\sdot)$, and $(\chi, \sdot, \psi)$, respectively.

As for $n = 3$, the graded quotient $I^3/I^4$ is a free $R$-module on $S_3$ of Lemma \ref{gradquot}.  
For each $s \in S_3$, we define a partial defining system $\phi_s$ (viewed as a pair of homomorphisms) as follows:
\begin{align*}
\phi_{x^3}\colon H \xrightarrow{\binmat{\chi}{3},0} \mr{U}_4(R) \times \mr{U}_1(R), \\
\phi_{xz}\colon H \xrightarrow{\mathrm{id},\chi} \mr{U}_3(R) \times \mr{U}_2(R), \\
\phi_{yx^2}\colon H \xrightarrow{\psi,\binmat{\chi}{2}} \mr{U}_2(R) \times \mr{U}_3(R), \\
\phi_{y^2x}\colon H \xrightarrow{\binmat{\psi}{2},\chi} \mr{U}_3(R) \times \mr{U}_2(R), \\
\phi_{y^3}\colon H  \xrightarrow{\binmat{\psi}{3},0} \mr{U}_4(R) \times \mr{U}_1(R), \\
\phi_{yz}\colon H \xrightarrow{\psi,\mathrm{id}} \mr{U}_2(R) \times \mr{U}_3(R).
\end{align*}
By Theorem \ref{connecting}, each partial defining system corresponds to a collection of Massey products as follows:
\begin{align*}
\phi_{x^3} \longleftrightarrow (\chi,\chi,\chi,\sdot), \\
\phi_{xz} \longleftrightarrow (\chi,\psi,\sdot,\chi), \\
\phi_{yx^2} \longleftrightarrow (\psi,\sdot,\chi,\chi), \\
\phi_{y^2x} \longleftrightarrow (\psi,\psi,\sdot,\chi), \\
\phi_{y^3} \longleftrightarrow (\psi,\psi,\psi,\sdot), \\
\phi_{yz}\longleftrightarrow (\psi,\sdot,\chi,\psi). 
\end{align*}
 
For each $s \in S_n$ with $n \in \{2,3\}$, the diagram \eqref{eq:general diagram} becomes
$$
\begin{tikzcd}
0 \arrow[r] & T \otimes_R I^n/I^{n+1} \arrow[r] \arrow{d}{p_s} & T \otimes_R \Omega/I^{n+1} \arrow[r] \arrow{d}{p_s} & T \otimes_R \Omega/I^n \arrow[r] \arrow{d}{p_s} & 0 \\
0 \arrow[r] & T \arrow[r] & \mf{U}_s(T) \arrow[r] & \mf{U}'_s(T) \arrow[r] & 0,
\end{tikzcd}
$$
where the maps $p_s$ are induced by the map $\phi_s$ and we have used the shorthand $\mf{U}_s(T)$ for $\mf{U}_{\phi_s}(T)$ 
(and similarly for the quotients).
Note that $p_s \colon T \otimes_R I^n/I^{n+1} \to T$ is just the $R$-tensor product of the likewise-defined
$p_s \colon I^n/I^{n+1} \to R$ with the identity on $T$.  
The maps $p_s \colon I^n/I^{n+1} \to R$ for $s \in S_n$ form the dual basis to the $R$-basis $S_n$ of $I^n/I^{n+1}$.  
This can be seen by an omitted direct computation, proceeding as in the following example.

\begin{example}
Suppose that $n = 2$, and take $s = z \in S_2$.
Recall that $\phi_z \colon H \to \mr{U}_3(R)$ is given by the canonical surjection $A \to R$ on coefficients.  By definition of $p_z \colon \Omega/I^3 \to 
\mf{U}_{\phi_z}(R) = M_{3,1}(R)$ in Lemma \ref{mapnilpotent}, 
we have 
$$
	p_z([h]) = \phi_z(h)  \left( \begin{smallmatrix} 0 \\0\\1 \end{smallmatrix} \right) \in M_{3,1}(R)
$$ 
for all $h \in H$.
Recalling that $x+1$, $y+1$, and $z+1$ are the group elements of matrices as in \eqref{generators}, we compute
\begin{equation}
\begin{aligned}
&p_{z}(x^2)
=  \left( \begin{smallmatrix} 0 \\0\\1 \end{smallmatrix} \right) - 2 \left( \begin{smallmatrix} 0 \\0\\1 \end{smallmatrix} \right)  +  \left( \begin{smallmatrix} 0 \\0\\1 \end{smallmatrix} \right)  = 0,\\
& p_{z}(y^2)=
\left( \begin{smallmatrix} 0 \\2\\1 \end{smallmatrix} \right) - 2  \left( \begin{smallmatrix} 0 \\1\\1 \end{smallmatrix} \right) +  \left( \begin{smallmatrix} 0 \\0\\1 \end{smallmatrix} \right) = 0,\\
& p_{z}(yx) =
 \left( \begin{smallmatrix} 0 \\1\\1 \end{smallmatrix} \right) -  \left( \begin{smallmatrix} 0 \\1\\1 \end{smallmatrix} \right) +
  \left( \begin{smallmatrix} 0 \\0\\1 \end{smallmatrix} \right) -  \left( \begin{smallmatrix} 0 \\0\\1 \end{smallmatrix} \right)  = 0,
\\
& p_{z}(z) =  \left( \begin{smallmatrix} 1 \\0\\1 \end{smallmatrix} \right) -  \left( \begin{smallmatrix} 0 \\0\\1 \end{smallmatrix} \right) 
= \left( \begin{smallmatrix} 1\\0\\0\\
\end{smallmatrix}\right),
\end{aligned}
\end{equation}
and note that $\left( \begin{smallmatrix} 1\\0\\0\\
\end{smallmatrix}\right)$ gives the identity of $R \subset \mf{U}_{\phi_z}(R)$.
\end{example}

By Theorem \ref{prop:general}, we then have the following.

\begin{theorem} \label{heis_image}
For $n \in \{2,3\}$ and $f \in Z^1(G,T \otimes_R \Omega/I^n)$, the element $\Psi^{(n)}([f])$ of
$$
	H^2(G,T) \otimes_R I^n/I^{n+1} \cong \bigoplus_{s \in S_n} H^2(G,T) s
$$ 
is the sum
\begin{equation}
\label{eq:Psi2 image}
(\chi,\chi,\lambda)_{\rho_{x^2}} x^2 + (\chi,\lambda,\psi)_{\rho_{yx}} yx + (\psi,\psi,\lambda)_{\rho_{y^2}}y^2 + (\chi,\psi,\lambda)_{\rho_{z} }z
\end{equation}
for $n = 2$ and the sum
\begin{align}
\begin{split}
\label{eq:Psi3 image}
(\chi,\chi,&\chi,\lambda)_{\rho_{x^3}} x^3 + (\chi,\psi,\lambda,\chi)_{\rho_{xz}} xz + (\psi,\lambda,\chi,\chi)_{\rho_{yx^2}}yx^2 \\
&+ (\psi,\psi,\lambda,\chi)_{\rho_{y^2x} }y^2x+ (\psi,\psi,\psi,\lambda)_{\rho_{y^3} }y^3+ (\psi,\lambda,\chi,\psi)_{\rho_{yz} }yz
\end{split}
\end{align}
for $n = 3$,
where each $\rho_s$ for $s \in S_n$ is the proper defining system relative to $\phi_s$ attached to $p_s \circ f$ by Lemma \ref{correspondence},
and $\lambda$ is the image of $f$ in $Z^1(G,T)$.
\end{theorem}

As before, Theorem \ref{augfilt} then provides the following isomorphisms.

\begin{corollary} \label{heis_cor}
	Suppose that $G$ is $p$-cohomologically finite of $p$-cohomological dimension $2$.
	For $n \in \{2,3\}$, 
	let $P_n(H)$ denote the subgroup of $H^2(G,T) \otimes_R I^n/I^{n+1}$ consisting of all sums in \eqref{eq:Psi2 image}
	for $n = 2$ and in \eqref{eq:Psi3 image} for $n = 3$.  We then  have a canonical isomorphism of $R$-modules
	$$ 
	\frac{I^nH^2_{\Iw}(N,T)}{I^{n+1} H^2_{\Iw}(N,T)} \cong \frac{H^2(G,T) \otimes_R I^n/I^{n+1}}{P_n(H)}.
	$$
\end{corollary}

\section{Applications to cyclotomic fields}\label{sec:cyclo}

In this section, we apply our general results to study class groups of finite extensions of a cyclotomic field $\Q(\mu_p)$ with $p$ an irregular prime. That is, under assumptions that include vanishing of certain cup products, we are able to bound the sizes of the $p$-parts of the class groups from below. We satisfy ourselves with describing a particularly clean setting of $p$-ramified $p$-extensions of $\Q(\mu_p)$, wherein the $p$-parts of class groups can be directly identified with second cohomology groups. Rather general results in an Iwasawa-theoretic context may be obtained as in \cite[Section 4]{reciprocity}. We consider $p$-ramified bicyclic and Heisenberg extensions; some simpler examples over cyclic extensions can be gleaned from the Iwasawa-theoretic treatment given in \cite[Section 7]{sh-massey}. 

We note the existence of a variety of works on Massey products in Galois groups with restricted ramification and the structure of class groups from perspectives different than ours, ranging from the much earlier work of Morishita \cite{morishita2004} and Vogel \cite{vogelthesis} to the very recent preprint of Ahlqvist--Carlson \cite{AC2022} concerning Massey products in \'etale cohomology.

\subsection{Notation and preliminaries} 
\label{subsec:cyclo setup}
In this subsection, we recall some standard facts regarding the mod $p$ unramified outside $p$ cohomology of the $p$th cyclotomic field, for an odd prime $p$. Most of these may be found, for instance, in \cite{mcs}.  Let $\Cl_K$ denote the ideal class group of a number field $K$. Let $S$ denote the set of primes over $p$ in any number field. Let $\Cl_{K,S}$ denote the $S$-class group of $K$, which is to say the class group of the ring $\mc{O}_{K,S}$ of $S$-integers of $K$. Let $G_{K,S}$ denote the Galois group of the maximal unramified outside $S$, or \emph{$p$-ramified}, extension of $K$. 

For any number field $K$ and prime $p$, Kummer theory provides an exact sequence
\[
0 \to \mc{O}_{K,S}^\times \otimes_{\Z} \F_p \to H^1(G_{K,S},\mu_p) \to \Cl_{K,S}[p] \to 0
\]
and a 
canonical injection
\[
\Cl_{K,S}  \otimes_{\Z} \F_p \xhookrightarrow{} H^2(G_{K,S},\mu_p)
\]
of $\F_p[\Delta]$-modules. The latter injection is an isomorphism if $K$ is a $p$-ramified, purely imaginary extension of $\Q$ with a unique prime over $p$.  
We shall write Massey products of elements of $H^1(G_{F,S},\mu_p)$ as products of elements of $F^{\times}/F^{\times p}$ whose Kummer cocycles (in this case characters) give classes in $H^1(G_{F,S},\mu_p)$, as opposed to the cocycles themselves. 

Now let $F = \Q(\zeta_p)$ for an odd prime $p$ and a primitive $p$th root of unity $\zeta_p$.
Note that $\mc{O}_{F,S} = \Z[\zeta_p,\frac{1}{p}]$ and $\Cl_{F,S} = \Cl_F$, since the prime $(1-\zeta_p)$ over
$p$ is principal.
Let $\Delta = \Gal(F/\Q)$, and let $\omega \colon \Delta \to \Z_p^\times$ be unique lift of the mod $p$ cyclotomic character. 
For $j \in \Z$, the $\omega^j$-isotypical component, or \emph{eigenspace}, of a $\Z_p[\Delta]$-module $M$ is
$$
	M\up{j} = \{ m \in M \mid \delta m = \omega(\delta)^j m \text{ for all } \delta \in \Delta \}.
$$ 

We say that a positive even integer $k < p$ is an \emph{irregular index} for $p$ if $\Cl_F[p]\up{1-k} \ne 0$, or equivalently, $p$ divides the numerator of the $k$th Bernoulli number $B_k$. As $p$ divides the denominator of $B_{p-1}$, every irregular index $k$ for $p$ satisfies $k \le p-3$.

We suppose that $p$ satisfies Vandiver's conjecture that $\Cl_{\Q(\zeta_p+\zeta_p^{-1})}[p] = 0$. By Leopoldt's reflection principle, this implies that for each irregular index $k$, the eigenspace $\Cl_F[p]\up{1-k}$ is cyclic, so we fix a generator and let $\alpha_k \in H^1(G_{F,S},\mu_p)^{(1-k)}$ be its unique lift.  
This also allows us to identify $H^2(G_{F,S},\mu_p)\up{1-k}$ with $\F_p$ via the isomorphisms
\[
\F_p \xrightarrow{1 \mapsto \alpha_k} H^1(G_{F,S},\mu_p)^{(1-k)} \xrightarrow{\sim} \Cl_F[p]\up{1-k} \xrightarrow{\sim} (\Cl_F \otimes_\Z \F_p)^{(1-k)} \xrightarrow{\sim} H^2(G_{F,S},\mu_p)\up{1-k},
\]
where the isomorphism $(\Cl_F \otimes_\Z \F_p)^{(1-k)} \xrightarrow{\sim} \Cl_F[p]\up{1-k}$ is multiplication by a power of $p$.

For an odd integer $i$, we define
\[
\eta_i \in (\mc{O}_{F,S}^\times \otimes_{\Z} \F_p)\up{1-i}
\]
to be the projection of $1-\zeta_p$ into that eigenspace. We often refer to the index $i$ as taking values in $\Z/(p-1)\Z$.
Via Kummer theory, we identify $\eta_i$ with an element of 
\[
H^1(G_{F,S},\mu_p)\up{1-i} \cong (H^1(G_{F,S},\mu_p) \otimes_{\Z} \mu_p^{\otimes (i-1)})^\Delta \cong H^1(G_{F,S},\mu_p^{\otimes i})^\Delta
\]
and $\zeta_p$ with an element of $H^1(G_{F,S},\mu_p)^{(1)}$. 
Vandiver's conjecture for $p$ is equivalent to the statement that
every $\eta_i$ is nontrivial. We codify all this and a bit more in the following remark.

\begin{remark} \label{eiggens}
For any positive integer $j < p$, the eigenspace $H^1(G_{F,S},\mu_p)^{(1-j)}$ is cyclic, generated by the element
\begin{itemize}
	\item $\eta_j$ if $j$ is odd,
	\item $\zeta_p$ if $j = p-1$,
	\item $\alpha_j$ if $j$ is an irregular index,
\end{itemize}
and is trivial for all other $j$.
If $i$ is odd, then the cup product with $\eta_i$ vanishes on $H^1(G_{F,S},\mu_p)$ if and only if $\eta_i \cup \eta_{k-i} = 0$ for all irregular indices $k$ for $p$.
\end{remark}

Given a $p$-extension $L/F$ that is unramified outside $p$ and for which $L/\Q$ is Galois, we can consider its Galois group $H = \Gal(L/F)$, which is of course normal inside $\Gal(L/\Q)$. Set $\Omega = \F_p[H]$ as before. We have an action of $G_{\Q,S}$ on $\Omega$ such that $g \in G_{\Q,S}$ sends the group element $[h]$ of $h \in H$ to $[\bar{g}h\bar{g}^{-1}]$, where $\bar{g}$ is the image of $g$ in $G$.

Since $G_{F,S}$ is normal in $G_{\Q,S}$, this $G_{\Q,S}$-action (together with right conjugation on $G_{F,S}$ in the usual fashion) induces a $\Gal(L/\Q)$-action on $H^*(G_{F,S},\mu_p \otimes_{\F_p} I^n/I^m)$ for every $0 \le n < m$. For $m = n+1$, this action factors through $\Delta$. We fix a lift of $\Delta$ to a subgroup of $\Gal(L/\Q)$ so that we may speak of the $\Delta$-action on these cohomology groups for all $n < m$, though in general this action depends upon the choice of lift. The generalized Bockstein maps $\Psi^{(n)}$ are then $\Delta$-equivariant.

\subsection{Class groups of bicyclic and Heisenberg extensions}

In this subsection, we let $i, j < p$ be distinct odd positive integers and set $K = F(\eta_i^{1/p},\eta_j^{1/p})$, with the slight abuse of notation
that we are in fact taking $p$th roots of any lifts of $\eta_i$ and $\eta_j$.  Note that $K/\Q$ is Galois. We assume throughout this subsection that 
\begin{itemize}
	\item $\Cl_F[p]$ is cyclic, and
	\item the cup products $\eta_i \cup \eta_{k-i}$ and $\eta_j \cup \eta_{k-j}$ vanish.
\end{itemize}
By the first assumption, $p$ has a unique irregular index $k$ and Vandiver's conjecture holds for $p$. In particular, $K$ is an $\F_p^2$-extension
of $F$. The interested reader might calculate how the bounds we give are worsened as one weakens these assumptions.

For $h \in \Z/(p-1)\Z$, let 
\[
	\delta_h = \begin{cases} 1 & \text{if } h \in \{0,k\}, \\ 0 & \text{otherwise}. \end{cases}
\]

\begin{proposition}\label{prop:bicycbound} 
We have
\[
\dim_{\F_p} H^2(G_{K,S},\mu_p) \ge 6-\delta_{2i}-\delta_{i+j}-\delta_{2j}.
\]
\end{proposition}

\begin{proof}
We will apply the results of Section \ref{bicyc_case} in the case that $G=G_{F,S}$ and $H=\Gal(K/F)$. 
To construct our lower bound, we will use the fact that
$\dim_{\F_p} H^2(G_{K,S},\mu_p) \ge \sum_{n=0}^2 d_n$,
where
$$
	d_n = \dim_{\F_p} \frac{I^nH^2(G_{K,S},\mu_p)}{I^{n+1}H^2(G_{K,S},\mu_p)}.
$$
So, first note that
\[
\frac{H^2(G_{K,S},\mu_p)}{IH^2(G_{K,S},\mu_p)} \cong  H^2(G_{F,S},\mu_p) \cong (\Cl_{F} \otimes_{\Z} \F_p)^{(1-k)},
\]
and the latter group has $\F_p$-dimension 1 by our assumption of Vandiver's conjecture, so $d_0 = 1$.

Let $x$ and $y$ be the ordered basis of $I/I^2 \cong H$ that is Kummer dual to $\eta_i$ and $\eta_j$.
The quantity $(\eta_i \cup \lambda)x + (\eta_j \cup \lambda)y$ is zero for $\lambda$ one of the generators of 
$H^1(G_{F,S},\mu_p)$ listed in Remark \ref{eiggens} unless (perhaps) if $\lambda$ is one of $\eta_{k-i}$ or $\eta_{k-j}$, in which cases it equals
$(\eta_i \cup \eta_{k-i})x$ and $(\eta_j \cup \eta_{k-j})y$, respectively.
Thus, by Proposition \ref{prop:abelian} and Theorem \ref{augfilt} for $n = 1$, we have
\[
\frac{IH^2(G_{K,S},\mu_p)}{I^2H^2(G_{K,S},\mu_p)} \cong \frac{H^2(G_{F,S},\mu_p) \otimes_{\F_p} I/I^2}{\langle(\eta_i \cup \eta_{k-i})  x , (\eta_j \cup \eta_{k-j} )y\rangle },
\]
and given the vanishing of the cup products on the right, we see that $d_1 =  2$.

Theorem \ref{augfilt} tells us that $d_2 = \dim_{\F_p} \coker \Psi^{(2)}$. For $\tilde{\lambda} \in H^1(G_{F,S},\Omega/I^2 \otimes_{\F_p} \mu_p)$ with image $\lambda \in H^1(G_{F,S},\mu_p)$,
Corollary \ref{cor:bicyc} provides the explicit formula
\begin{equation} \label{eq:psi2}
	\Psi^{(2)}(\tilde{\lambda}) = (\eta_i,\eta_i,\lambda)_{\rho_{x^2}} x^2 + (\eta_i,\lambda,\eta_j)_{\rho_{xy}} xy 
	+ (\lambda,\eta_j,\eta_j)_{\rho_{y^2}} y^2. 
\end{equation}
Since cup products with $\eta_i$ and $\eta_j$ are trivial by assumption and Remark \ref{eiggens}, we see that the expression on the right
of \eqref{eq:psi2} is independent of the proper defining systems, and therefore $\Psi^{(2)}$ factors through $H^1(G_{F,S},\mu_p)$.

Now suppose that for some $h$ we have $\lambda \in H^1(G_{F,S},\mu_p)^{(1-h)}$, a space of dimension at most $1$. Note that $\Delta$ acts on $x^2$ through $\omega^{2i}$, on $xy$ through $\omega^{i+j}$, and on $y^2$ through $\omega^{2j}$. We then see by Remark \ref{eiggens} that the Massey products in \eqref{eq:psi2} can be nontrivial if and only if $h-2i$, $h-i-j$, or $h-2j$ (in that order) is congruent to $0$ or $k$ modulo $p-1$, which is to say if and only if $\delta_{2i} = 1$, $\delta_{i+j} = 1$, or $\delta_{2j} = 1$. 
Since $\dim_{\F_p} H^2(G_{F,S},\mu_p) \otimes_{\F_p} I^2/I^3 = \dim_{\F_p} I^2/I^3 = 3$, 
we have $d_2 \ge 3-\delta_{2i}-\delta_{i+j}-\delta_{2j}$, as required.
\end{proof}

Note that $\eta_i \cup \eta_j = 0$, since we must have $j \equiv k-i \bmod p-1$ for this cup product not to vanish, and we have assumed that
$\eta_i \cup \eta_{k-i} = 0$. Thus, there exists a degree $p$ extension $L$ of $K$, Galois over $F$ and unramified outside $p$, such 
that $\Gal(L/F) \cong \mr{U}_3(\F_p)$.  We can and do choose $L$ to be Galois over $\Q$: in fact,
\cite[Proposition 2.7]{sh-thesis} provides the following description of Kummer generators of such fields $L$, viewed as extensions of $K$.

\begin{remark}
	Set $E = F(\eta_i^{1/p})$, and let $\sigma$ be a generator of $\Gal(E/F)$. 
	Write $\eta_j = \prod_{i=0}^{p-1} \sigma^i \beta'$ for some $\beta' \in E^{\times}/E^{\times p}$.
	Pick a lift of $\Delta$ to a subgroup of $\Gal(E/\Q)$. Let $\beta$ be 
	the projection of $\beta'$ to the $\omega^j$-eigenspace of $E^{\times}/E^{\times p}$ for the action of
	this lift. The fact that $\eta_j \in (F^{\times}/F^{\times p})^{(j)}$ implies that $\eta_j = \prod_{i=0}^{p-1} \sigma^i \beta$ as well. 
	We then have $L = K((c\gamma)^{1/p})$ for $\gamma = \prod_{i=1}^{p-1} \sigma^i \beta^i$ and any
	$c \in H^1(G_{F,S},\mu_p)^{(1-i-j)}$.
	The latter group is zero if $\delta_{i+j} = 0$
	(see Remark \ref{eiggens}), in which case $L$ is unique.
\end{remark}

The group $\Delta$ acts on $\Gal(L/K)$ by conjugation through $\omega^{i+j}$. It may be helpful for the reader to view
$\Gal(L/\Q)$ as the group of matrices
\[
\begin{pmatrix}
\omega(\delta)^i & * & *\\
0 & 1 & *\\
0 & 0 & \omega(\delta)^{-j}
\end{pmatrix}
\]
for some $\delta \in \Delta$.

\begin{proposition}\label{prop:heisbound}
We have
\[
\dim_{\F_p} H^2(G_{L,S},\mu_p) \ge 7-\delta_{2i}-\delta_{2j}-\delta_{i+j}.
\]
\end{proposition}

\begin{proof}
We will apply the results of Section \ref{heis_case} in the case that $G=G_{F,S}$ and $H=\Gal(L/F)$. Set
$$
	d_n = \dim_{\F_p} \frac{I^nH^2(G_{L,S},\mu_p)}{I^{n+1}H^2(G_{L,S},\mu_p)}.
$$
We have $d_0 = 1$ and $d_1 = 2$ by the same arguments as for $K/F$ (noting that in the Heisenberg case, we still have $I/I^2 \cong H^{\ab}
\cong \F_p^2$). As in the proof of Proposition \ref{prop:bicycbound}, we must give a lower bound on $d_2 = \dim_{\F_p} \coker \Psi^{(2)}$.

Corollary \ref{heis_cor} tells us that for $\tilde{\lambda} \in H^1(G_{F,S},\Omega/I^2 \otimes_{\F_p} \mu_p)$ with image $\lambda \in H^1(G_{F,S},\mu_p)$, we have
\[
	\Psi^{(2)}(\tilde{\lambda}) = (\eta_i,\eta_i,\lambda)_{\rho_{x^2}} x^2 + (\eta_i,\lambda,\eta_j)_{\rho_{yx}} yx + 
	(\eta_j,\eta_j,\lambda)_{\rho_{y^2}} y^2 + (\eta_i,\eta_j,\lambda)_{\rho_z} z.
\]
Again, the vanishing of $\eta_i \cup \eta_{k-i}$ and $\eta_j \cup \eta_{k-j}$ ensures that $\Psi^{(2)}(\tilde{\lambda})$ depends only on $\lambda$. As before, but now noting also that $\Delta$ acts on $z \in I^2/I^3$ by $\omega^{i+j}$, we see that these Massey products
must vanish unless $\delta_{2i} = 1$, $\delta_{i+j} = 1$, $\delta_{2j} = 1$, and $\delta_{i+j} = 1$, respectively. Moreover, if $\delta_{i+j} = 1$,
then the image of $\Psi^{(2)}$ on $H^1(G_{F,S},I/I^2 \otimes_{\F_p} \mu_p)^{(k-i-j)}$ is at most one-dimensional, generated by $(\eta_i,\lambda,\eta_j)yx + (\eta_i,\eta_j,\lambda)z$ for $\lambda = \zeta_p$ or $\lambda = \alpha_k$, by Remark \ref{eiggens} (and similarly for the other cases). Thus, we have $d_2 \ge 4 - \delta_{2i} - \delta_{2j} - \delta_{i+j}$.
\end{proof}

\begin{remark} \label{rem:pairings} \
	\begin{enumerate}
		\item If $2i \equiv k \bmod p-1$, so in particular $\delta_{2i} = 1$, 
		then the condition that $\eta_i \cup \eta_i = 0$ is automatic by antisymmetry of the cup product. 
		\item It occurs that $\delta_{2i}$, $\delta_{2j}$, and $\delta_{i+j}$ are all $1$ if and only if $p$ is $1$ modulo
		$4$ but not $8$ and $k = \frac{p-1}{2}$, so that we have $\{i,j\} = \{\frac{p-1}{4},\frac{3(p-1)}{4}\}$. 
		
		We can have $2i \equiv 0 \bmod p-1$ only if $p \equiv 3 \bmod 4$, in which case $i = \frac{p-1}{2}$. We
		can then also choose $j$ such that $2j \equiv k \bmod p-1$ (see
		Example \ref{ex:twodeltas}), but then $i+j$ is either $\frac{k}{2}$ or $\frac{k}{2}+\frac{p-1}{2}$
		modulo $p-1$, which cannot be $0$ or $k$, so $\delta_{i+j} = 0$.
		
		\item The $p$th root of $\eta_{p-k}$ generates the unique degree $p$ unramified extension of $F$, and it satisfies
		$\eta_{p-k} \cup \eta_{2k-1} = 0$. In such a setting, $\dim_{\F_p} H^2(G_{F(\eta_{p-k}^{1/p}),S},\mu_p) = p-1$, coming
		entirely from the Brauer part of this second cohomology group. 
		
		On the other hand, suppose that $i$ and $j$ are not $p-k$ modulo $p-1$. (By
		what we have just said, we can take one of them to be $2k-1$, so long as $2k-1$ is not $p-k$ modulo $p-1$, i.e.,  
		$3k \not\equiv 2 \bmod p-1$.)
		Then $K/\Q$ is totally ramified at $p$, which forces $L/\Q$ to be as well. 
		This implies that
		\begin{eqnarray*} 
			H^2(G_{K,S},\mu_p) \cong \Cl_{K,S} \otimes \F_p &\mr{and}& H^2(G_{L,S},\mu_p) \cong \Cl_{L,S} \otimes \F_p.
		\end{eqnarray*} 
		In particular, our lower bounds on the $\F_p$-dimensions of these $S$-class groups 
		give lower bounds on the dimensions of the class groups $\Cl_K \otimes \F_p$ and $\Cl_L \otimes \F_p$.
	\end{enumerate}
\end{remark}

We conclude with some numerical examples. Many more are available using the tables referenced in \cite{mcs}, which compute
the cup product pairings up to scalar for primes less than 25,000.

\begin{example} \label{ex:twodeltas}
Let $p=59$, for which $k = 44$ is the unique irregular index. For $i=29$ and $j=51$, we have $2i \equiv 0 \bmod p-1$ and $2j \equiv k \bmod
p-1$, so $\delta_{2i} = \delta_{2j}=1$, while $\delta_{i+j} = 0$. We then have $\eta_j \cup \eta_{k-j} = \eta_{51} \cup \eta_{51} =0$ and $\eta_i \cup \eta_{k-i} = \eta_{29} \cup \eta_{15} = 0$ as in Remark \ref{rem:pairings}, noting that $15 = p-k$.
Given this, Propositions \ref{prop:bicycbound} and \ref{prop:heisbound} provide the following lower bounds on the $p$-ranks of the class groups of $K$ and $L$:
\begin{eqnarray*}
	\dim_{\F_p} \Cl_K \otimes \F_p \ge 4 &\mr{and}& \dim_{\F_p}\Cl_L \otimes \F_p \ge 5.
\end{eqnarray*}
The relevant, potentially nonzero, Massey triple products in this example are $(\eta_{51},\eta_{51},\zeta_{59})$ and $(\eta_{29},\eta_{29},\alpha_{44})$.
\end{example}

\begin{example} \label{ex:onedelta}
Let $p=67$, for which $k = 58$ is the unique irregular index. For $i = 29$ and $j = 49$, we have $\delta_{2i} = 1$ and $\delta_{2j} = \delta_{i+j} = 0$. Since $p-k = 9 \notin \{29,49\}$, we have 
\begin{eqnarray*}
	\dim_{\F_p} \Cl_K \otimes \F_p \ge 5 &\mr{and}& \dim_{\F_p}\Cl_L \otimes \F_p \ge 6.
\end{eqnarray*}
Here, the interesting Massey product is $(\eta_{29},\eta_{29},\zeta_{67})$.
\end{example}

It is not hard to find examples in which all the error terms vanish, so the maximal lower bounds are achieved.

\begin{example}
Let $p=101$, which has unique irregular index $k = 68$. Take $i=13$ and $j=35$. The computations referenced in \cite{mcs} show that $\eta_{13} \cup \eta_{55}=0$, and $\eta_{35} \cup \eta_{33}=0$ holds since $p-k = 33$. Since $\delta_{i+j}=\delta_{2i}=\delta_{2j}=0$, we have
\begin{eqnarray*}
	\dim_{\F_p} \Cl_K \otimes \F_p \ge 6 &\mr{and}& \dim_{\F_p}\Cl_L \otimes \F_p \ge 7.
\end{eqnarray*}
The same lower bounds are achieved for $\{i,j\} = \{35,55\}$. (Note that $\{i,j\} = \{13,55\}$ has $\delta_{i+j} = 1$, so the bounds for this
pair are one worse.)
\end{example}

Notice that genus theory has no contribution to the lower bounds (for $S$-class groups) in the above examples. Indeed, an unramified $\F_p$-extension of either $K$ or $L$ which descends to an abelian extension of $F$ would contribute to the zeroth graded piece in the augmentation filtration, but in these examples this is entirely accounted for the class group of $F$ (i.e., all such extensions are already unramified over $F$).

\section{Massey vanishing for absolute Galois groups} \label{masseyvanish}

In this final section of this paper, we apply our techniques to study absolute Galois groups of fields. The motivating problem is to determine which profinite groups can be isomorphic to the absolute Galois group $G_F$ of a field $F$.  Artin and Schreier showed in 1927 that any nontrivial finite group with this property is the cyclic group of order two.  Other restrictions are reflected in the cohomological properties of $G_F$.

The norm residue isomorphism theorem, or Milnor-Bloch-Kato conjecture, proven by Voevodsky and Rost (see \cite{voevodsky}), tells us that the algebra $H^*(G_F,\F_p)$ under cup product is isomorphic to the mod-$p$ Milnor $K$-theory of $F$ (for $F$ containing a primitive $p$th root of $1$). In particular, this implies that the $\F_p$-cohomology algebra is generated in degree $1$ with all relations generated in degree $2$.

Going beyond cup products to higher cohomological operations, Min\'a\v{c} and T\^an formulated a remarkable conjecture, known as the \emph{Massey vanishing conjecture}, for Massey products of $\fp$-valued characters on the absolute Galois group $G_F$ of a field $F$ in \cite{mt1}.  For $n \ge 3$, it states that any $n$-fold Massey product of characters $G_F \to \fp$ that has a defining system has some defining system for which the resulting Massey product is zero. The Massey product is said to \emph{contain zero} if such a defining system exists.  As evidence for this, Efrat--Matzri \cite{em} and Min\'a\v{c}--T\^an \cite{mt2} independently proved \emph{triple Massey vanishing}, which is to say the conjecture for $n = 3$ and arbitrary $p$. 

The Massey vanishing conjecture was inspired by work of Hopkins--Wickelgren \cite{hw}: using splitting varieties, they had proven that $3$-fold Massey products over number fields that are defined contain zero when $p=2$ \cite{hw}.
Massey vanishing over number fields was extended to successively general $n$ for arbitrary primes $p$: to $n=3$ in \cite{mt3}, to $n=4$ in \cite{gmt}, and to all $n$ in work of Harpaz--Wittenberg \cite{HW}. In each of these cases, the method is specific to number fields because it uses a local-to-global principal to prove the existence of rational points on a splitting variety.
For local fields, the Massey vanishing conjecture is known due to \cite{mt1}.

The differential graded ring $C(G_F,\F_p)$ of continuous $\F_p$-valued $G_F$-cochains is said to be \emph{formal} if it is quasi-isomorphic to $H^*(G_F,\F_p)$. The question of whether or not $C(G_F,\F_p)$ is always formal was raised by Hopkins and Wickelgren in their aforementioned work and answered in the negative by Positselski \cite{positselski2017}. If formality holds for some $F$ and $p$, then a stronger version of Massey vanishing, that moreover the vanishing of the consecutive cup products yields definedness, holds in that instance. 
Pal and Quick \cite{PQ2022} have recently shown that if $G_F$ is real projective (e.g., has virtual cohomological dimension at most $1$), then
$C(G_F,\F_p)$ is in fact formal. Also very recently, Quadrelli \cite{Q} showed that if $G$ is a pro-$p$ group of elementary type, then $G$ has the strong Massey vanishing property, which applies to several classes of fields.

The latter two results suppose a condition on the structure of $G_F$. Other results tend to require that several of the characters in the Massey products be the same. For instance, the third author had long ago proved in \cite{sh-massey} what we refer to here as the \emph{$p$-cyclic Massey vanishing property} for absolute Galois groups of fields containing a primitive $p$th root of unity: for $n \le p-1$, all definable $(n+1)$-fold Massey products with identical first $n$ entries vanish with respect to some proper defining system (i.e., $(\chi^{(n)},\psi)$ contains $0$ for $\chi, \psi \in H^1(G_F,\F_p)$ with $\chi \cup \psi = 0$).
Beyond this, Min\'a\v{c} and T\^an \cite{mt rigid} proved the vanishing of $n$-fold Massey products when all $n$ characters are the same, for arbitrary $n$ and fields containing $2p$th roots of unity. In sufficiently large characteristic, Efrat proved the vanishing of $n$-fold Massey products with all entries coming from either $z$ or $1-z$ for a fixed field element $z \in F^{\times} - \{1\}$, improving upon a result of Wickelgren \cite{wickelgren}. In another very recent preprint, Merkurjev and Scavia \cite{MS2022} prove that quadruple Massey products with the same first and last entries vanish for $p = 2$ for $F$ of characteristic not $2$.

As should be expected, the Massey vanishing conjecture has strong implications for the structure of absolute Galois groups. For instance, it often allows for the realization of nilpotent field extensions: we mention \cite{GM19} as an example of a recent work in this direction.

\subsection{The cyclic Massey vanishing property}
\begin{definition} \label{defn:cyclic massey vanishing}
Let $G$ be a profinite group, and let $p$ be a prime number. We say that $G$ has the \emph{$p$-cyclic Massey vanishing property}
if for all homomorphisms $\chi, \lambda \colon G \to \F_p$ with $\chi \cup \lambda = 0$, there exists a proper defining system such that
$(\chi^{(p-1)},\lambda)$ vanishes.
\end{definition}

As a simple corollary of \cite[Theorem 4.3]{sh-massey}, the absolute Galois group of field $F$ containing a primitive $p$th root of unity has the $p$-cyclic Massey vanishing property.  (For this, consider the case that $\Omega$ is the separable closure of $K$ and $m=1$ in the notation of said theorem.)  The proof uses only the fact that if the norm residue symbol $(a,b)_{p,F}$ vanishes, then $b$ is a norm from $F(a^{1/p})$. We shall give a streamlined proof of this and more, using the following abstract characterization of a standard property of absolute Galois groups.

\begin{definition} \label{defn:absGaltype}
	Let $m \ge 1$, and set $R = \Z/m\Z$.  We say that a profinite group $G$ is of 
	\emph{$m$-absolute Galois type} if it has the property that, for any $\chi \in H^1(G,R)$, the sequence
	\begin{equation} \label{eq:galoistype}
		H^1(G,R[H_\chi]) \to H^1(G,R) \xrightarrow{\chi\, \cup} H^2(G,R) \to H^2(G,R[H_\chi])
	\end{equation}
	is exact, where $H_\chi = G /\ker(\chi)$ is the coimage of $\chi$.
\end{definition}

Under Shapiro's lemma, the first and last maps in \eqref{eq:galoistype} are identified with corestriction and
restriction maps, respectively \cite[Proposition 1.6.5]{nsw}. It is well-known that an absolute Galois group $G_F$ is of $m$-absolute Galois type if $F$ contains a primitive $m$th root of unity (see for instance \cite[Propositions XIV.2 and XIV.4]{serre}).
This condition on $G_F$ generalized to arbitrary cohomological degree is heavily used in the proof of the norm residue isomorphism theorem: see \cite[Theorem 3.6]{hawe}.  We focus on the comparison of $p$-absolute Galois type with $p$-cyclic Massey vanishing.  In fact, our results would allow us to prove a more general but analogous result for profinite groups with the property that characters on $G$ of order $p^s$ lift to characters of order $p^t$ for some large enough $t$ relative to $s$, under conditions as in Section \ref{procyclic}.

\begin{remark}
It is known that there exist groups that of $p$-absolute Galois type that are not isomorphic to the absolute Galois group of any field \cite{BCQ}.
\end{remark}

\begin{proposition} \label{pcyclic}
	Let $G$ be a profinite group. Then $G$ has the $p$-cyclic Massey vanishing property if and only if the sequence
	\eqref{eq:galoistype} is exact at $H^1(G,\F_p)$.  
\end{proposition}

\begin{proof}
    	Let $\chi, \lambda \colon G \to \F_p$ with $\chi \cup \lambda = 0$, and set $\Omega = \F_p[H_{\chi}]$. The $p$th power
	of the augmentation ideal in $\Omega$ is zero, and the kernel of the generalized Bockstein map $\Psi^{(p-1)}$ is the image of
	$H^1(G,\Omega) \to H^1(G,\Omega/I^{p-1})$.  
	Theorem \ref{cyclic_case} tells us that the Massey product $(\chi^{(p-1)},\lambda)$ is defined and vanishes
	for some choice of proper defining system in $H^1(G,\Omega/I^{p-1})$ if and only if $\lambda$ lifts to $H^1(G,\Omega)$.  	
	From this, we have the proposition.
\end{proof}

Proposition \ref{pcyclic} applies in particular to the absolute Galois group of any field $F$ containing a primitive $p$th root of unity, i.e., $G_F$ has the $p$-cyclic Massey vanishing property.  We also have the following result, which may be of independent interest.

\begin{proposition}
	Let $G$ be a profinite group.  If \eqref{eq:galoistype} is exact at $H^2(G,\F_p)$ for a given $\chi \in H^1(G,\F_p)$, then it is
	exact at $H^1(G,\F_p)$, so $G$ is of $p$-absolute Galois type.
\end{proposition}

\begin{proof}    
	Let $\chi, \lambda \colon G \to \F_p$ with $\chi \cup \lambda = 0$, and
    	suppose that \eqref{eq:galoistype} is exact at $H^2(G,\F_p)$. 
	We have to show that there is a proper defining system $\rho$
	such that $(\chi^{(p-1)},\lambda)_{\rho}$ vanishes. We may suppose that $\chi \neq 0$.  Let $x = [h]-1$ for $h \in H_{\chi}$
	with $\chi(h) = 1$.  By induction on $n$, we can assume that there is a proper 
	defining system $\rho_{x^n}$ for $(\chi^{(n)},\lambda)$ with $n < p$ determined by some 
	$f = \sum_{k=0}^{n-1} \lambda_k x^k \in Z^1(G, \Omega/I^n)$, with $\lambda$ necessarily equal to $\lambda_0$.  
	Writing $(\chi^{(n)},\lambda)_f$ for the corresponding Massey product $(\chi^{(n)},\lambda)_{\rho_{x^n}}$, we have
    	\[
    		(\chi^{(n)},\lambda)_f = \chi \cup \lambda_{n-1} + \binom{\chi}{2} \cup \lambda_{n-2} + \dots + \binom{\chi}{n} \cup \lambda.
    	\]
    	Clearly the restriction of $(\chi^{(n)},\lambda)_f$ to $\ker(\chi)$ vanishes, so, by the exactness of \eqref{eq:galoistype} at $H^2(G,\F_p)$, 
	we have $(\chi^{(n)},\lambda)_f  =\chi \cup \psi$ for some $\psi \in H^1(G,\F_p)$. Then we see that 
	$f'= f - \psi x^{n-1}$ is a proper defining system such that the Massey product $(\chi^{(n)},\lambda)_{f'}$ vanishes.
	By Theorem \ref{cyclic_case}, this implies that the class of $f'$ is in the kernel of $\Psi^{(n)}$, so it lifts to the class of some 
	$\tilde{f} \in Z^1(G,\Omega/I^{n+1})$,
	which gives rise to a proper defining system $\rho_{x^{n+1}}$ for $(\chi^{(n+1)},\lambda)$.  If $n+1 = p$, then the class of
	$\tilde{f}$ is the desired lift to $H^1(G,\Omega)$.
\end{proof}

It is unclear that exactness of \eqref{eq:galoistype} at $H^1(G,\F_p)$ should imply exactness at $H^2(G,\F_p)$.  

\subsection{Triple Massey vanishing}

In this subsection, let us suppose that $p$ is an odd prime.
The following theorem gives a new proof of the vanishing of Massey triple products for absolute Galois groups due to Efrat--Matzri \cite{em} and Min\'a\v{c}--T\^an \cite{mt2}.  Both proofs utilized the fact that the absolute Galois groups of a field containing a primitive $p$th root of unity are of $p$-absolute Galois type.  We show that the potentially weaker condition of $p$-cyclic Massey vanishing suffices.

\begin{theorem}
\label{thm:triple vanishing}
	Let $G$ be a profinite group with the $p$-cyclic Massey vanishing property for an odd prime $p$.
	Let $\chi,\psi,\lambda \in H^1(G,\F_p)$ be such that 
	$\chi \cup \lambda = \lambda \cup \psi = 0$. Then there exists a defining system $\rho$ for $(\chi,\lambda,\psi)$ 
	such that the Massey triple product $(\chi,\lambda,\psi)_{\rho}$ vanishes.
\end{theorem}

The case where $\chi$ and $\psi$ are linearly dependent follows easily from the $p$-cyclic Massey vanishing property, so we can and do assume that $(\chi,\psi) \colon G \to \F_p^2$ is surjective, and we let $H$ be the coimage. Let $\Omega=\F_p[H]$ and let $I \subset \Omega$ be the augmentation ideal. Let $h_\chi, h_\psi \in H$ be the dual basis to $(\chi,\psi)$, and let $x=[h_\chi]-1$ and $y=[h_\psi]-1$ so that $I=x\Omega+y\Omega$.

We want to make maximal use of the fact that $G$ has the cyclic Massey vanishing property. For this, we let $C_1$, $C_2$, and $C_3$ be the coimages of $\alpha_1=\chi$, $\alpha_2=\psi$, and $\alpha_3=\chi+\psi$, respectively. 
Let $\Omega_i=\F_p[C_i]$, and let $I_i \subset \Omega_i$ be its augmentation ideal. Let $\gamma_i \in C_i$ with $\alpha_i(\gamma_i) = 1$, and let $x_i = [\gamma_i]-1 \in I_i$. Note that each $\alpha_i$ factors through $H$, so $\alpha_i$ induces a surjective ring homomorphism $\Omega \to \Omega_i$ that we also call $\alpha_i$. Then note that
\begin{equation}\label{eq:alpha of variables}
(\alpha_1(x),\alpha_1(y)) = (x_1,0), \quad (\alpha_2(x),\alpha_2(y)) = (0,x_2), \quad  (\alpha_3(x),\alpha_3(y)) = (x_3, x_3).
\end{equation}

Now consider the ideal $J = I^3 + xy\Omega$, and let $J_i = \alpha_i(J)$. By \eqref{eq:alpha of variables}, we have $J_1=I_1^3$, $J_2=I_2^3$, and $J_3=I_3^2$. Hence we have a commutative diagram with exact rows
\begin{equation}
	\label{eq:main diagram}
	\begin{tikzcd}
	0 \ar[r] & J/I^3 \ar[r] \ar[d,"\wr"] & I/I^3 \ar[r] \ar[d, hook] & I/J \ar[r] \ar[d, hook]& 0 \\ 
	0 \ar[r] & I_3^{2}/I_3^3 \ar[r] & \bigoplus_{i=1}^3 I_i/I_i^3 \ar[r] &  \bigoplus_{i=1}^3 I_i/J_i \ar[r] 
& 0,
	\end{tikzcd}
\end{equation}
where the vertical maps are induced by the maps $\alpha_i$. Note that $J/I^3 = \F_p xy$, so the leftmost vertical arrow is an isomorphism, and $I/J= \F_p x \oplus \F_p x^2 \oplus  \F_p y \oplus \F_p y^2$, and the map $I/J \to I_1/I_1^3 \oplus I_2/I_2^3$ is an isomorphism, so the rightmost vertical arrow is split injective.

\begin{lemma} \label{commdiag}
	There is a commutative diagram with exact rows
	\begin{equation}
	\label{eq:cohom of main}
	\begin{tikzcd}
		H^1(G,I/J) \ar[r] \ar[d]  & H^2(G,J/I^3) \ar[r,"\iota"] \ar[d,"\wr","f"'] & H^2(G,I/I^3) \ar[d,"g"'] \ar[r] & H^2(G,I/J) \ar[d, hook] \\
		H^1(G, \F_p) \ar[r,"\alpha_3 \, \cup"] & H^2(G,\F_p) \ar[r,"h"] & \bigoplus_{i=1}^3 H^2(G,I_i/I_i^3) \ar[r] & \bigoplus_{i=1}^3 H^2(G,I_i/J_i),
	\end{tikzcd}
	\end{equation}
	where $f$ is the isomorphism $\xi \cdot xy \mapsto \xi$ and $g$ is the map induced by the center vertical arrow in \eqref{eq:main diagram}.
\end{lemma}

\begin{proof}
	The lower sequence in \eqref{eq:main diagram} is a direct sum of three exact sequences for $i \in \{1,2,3\}$, where for $i \in \{1,2\}$ 
	the sequence has zero as its first term.  Taking cohomology of \eqref{eq:main diagram}, we obtain the commutative diagram with exact rows
	\begin{equation}
	\begin{tikzcd}[column sep = small]
		H^1(G,I/J) \ar[r] \ar[d]  & H^2(G,J/I^3) \ar[r,"\iota"] \ar[d,"\wr"] & H^2(G,I/I^3) \ar[d,"g"'] \ar[r] & H^2(G,I/J) \ar[d,hook] \\
		\bigoplus_{i=1}^3 H^1(G,I_i/J_i)\ar[r,"\partial_3"] & H^2(G,I_3^2/I_3^3) \ar[r] & \bigoplus_{i=1}^3 H^2(G,I_i/I_i^3) \ar[r] & \bigoplus_{i=1}^3 		H^2(G,I_i/J_i),
	\end{tikzcd}
	\end{equation}
	where $\partial_3$ is zero on the first two terms of the summand and the connecting map on the third.   Note that the rightmost
	vertical arrow is injective by the split injectivity of the underlying map on coefficients.
	To complete the proof, we have to show that $\partial_3(\beta) = \alpha_3 \cup \beta$ for $\beta \in H^1(G,I_3/J_3)$. 
	But the lower sequence in \eqref{eq:main diagram} for $i=3$ is isomorphic to
	\[
		0 \to I_3/I_3^2 \to \Omega_3/I_3^2 \to \F_p \to 0
	\]
	via the isomorphism $I_3/J_3 \xrightarrow{\sim} \F_p$ taking the image of $x_3$ to $1$, 
	so this follows from Proposition \ref{prop:abelian}.
\end{proof}

\begin{proof}[Proof of Theorem \ref{thm:triple vanishing}]
	Consider the commutative diagram of exact sequences
	\[
	\begin{tikzcd}
		0 \ar[r]  & J/I^3 \ar[r] \ar[d] & \Omega/I^3 \ar[r] \ar[d,equal] & \Omega/J \ar[r] \ar[d] & 0 \\
		0 \ar[r] & I/I^3 \ar[r] & \Omega/I^3 \ar[r] & \F_p \ar[r] & 0
	\end{tikzcd}
	\]
	and the associated diagram in cohomology
	\begin{equation}
	\label{eq:J and I}
	\begin{tikzcd}
		H^1(G, \Omega/I^3) \ar[d,equal] \ar[r] & H^1(G, \Omega/J) \ar[r,"\partial'"] \ar[d] & H^2(G, J/I^3) \ar[d,"\iota"] \ar[r] & H^2(G, \Omega/I^3) 
		\ar[d,equal] \\
		H^1(G, \Omega/I^3) \ar[r] & H^1(G,\F_p) \ar[r,"\partial"] & H^2(G,I/I^3) \ar[r] & H^2(G, \Omega/I^3),
	\end{tikzcd}
	\end{equation}
	where $\iota$ is as in \eqref{eq:cohom of main}.

	Now let $\lambda \in H^1(G,\F_p)$ be as in the statement of the theorem and consider the element $\partial(\lambda) \in H^2(G,I/I^3)$. 
	Then $g(\partial(\lambda)) \in \bigoplus_{i=1}^3 H^2(G,I_i/I_i^3)$ is the obstruction to lifting $\lambda$ to $H^1(G,\Omega/I_i^3)$ for 
	all $i \in \{1,2,3\}$, and this vanishes by the $p$-cyclic Massey vanishing property.  Hence $g(\partial(\lambda))=0$. 

	By Lemma \ref{commdiag} and the injectivity of the rightmost vertical arrow in \eqref{eq:cohom of main}, this implies that 
	$\partial(\lambda)$ is in the image of $\iota$.  By the commutativity of \eqref{eq:J and I}, there is then a lift 
	$\tilde{\lambda} \in H^1(G, \Omega/J)$ 
	of $\lambda$. Using Corollary \ref{cor:bicyclic}, we see that $\tilde{\lambda}$ 
	determines a proper defining system $\rho_{xy}$ for $(\chi,\lambda,\psi)$ such that
	 $\partial'(\tilde{\lambda}) =(\chi,\lambda,\psi)_{\rho_{xy}} \cdot xy$.

	By Lemma \ref{commdiag}, we have
	\[
		hf(\partial'(\tilde{\lambda})) = g\iota(\partial'(\tilde{\lambda}))=g(\partial(\lambda))=0,
	\]
	so $f(\partial'(\tilde{\lambda})) \in \ker(h) = \image(\alpha_3\, \cup)$. Hence we have
	\begin{equation}
	\label{eq:psi}
		f(\partial'(\tilde{\lambda})) = (\chi,\lambda,\psi)_{\rho_{xy}} = \alpha_3 \cup \nu = \chi \cup \nu - \nu \cup \psi
	\end{equation}
	in $H^2(G,\F_p)$, for some $\nu \in H^1(G,\F_p)$.  In particular, we have that
	$$
		(\chi,\lambda,\psi)_{\rho_{xy}} \in \image(\chi\, \cup) + \image(\cup\, \psi),
	$$ 
	which implies that there is a defining system $\rho$ such that $(\chi,\lambda,\psi)_{\rho} = 0$.
\end{proof}

The reader may note that in Theorem \ref{thm:triple vanishing}, we used something weaker than $p$-cyclic Massey
vanishing.  Namely, the actual condition employed is that for any character $\chi \colon G \to \F_p$, the sequence
\begin{equation} \label{weaker_exactness}
	H^1(G,\F_p[H_{\chi}]/I_{\chi}^3) \to H^1(G,\F_p) \xrightarrow{\chi \,\cup} H^2(G,\F_p)
\end{equation}
is exact, where $H_{\chi} = G/\ker(\chi)$ and $I_{\chi}$ is the augmentation ideal in $\F_p[H_{\chi}]$.  This is
equivalent to the statement that if $\chi \cup \lambda = 0$ for some $\lambda \in H^1(G,\F_p)$, then 
$(\chi,\chi,\lambda)$ is zero for some proper defining system.

\begin{remark}
	In \cite[Corollary 3.5]{matzri}, Matzri proved that triple Massey vanishing follows from defined Massey products of the form 
	$(\chi,\lambda,\chi)$ containing zero. The proof exploits the exactness of \eqref{eq:galoistype} at $H^2(G,\F_p)$
	to obtain this vanishing.  From our perspective, the vanishing of these Massey products follows directly from the exactness 
	of \eqref{weaker_exactness}.
\end{remark}

\begin{remark}
	The proof of Theorem \ref{thm:triple vanishing} does not show that $G$ has the ``bicyclic Massey vanishing property'' that 
	any $\lambda \in H^1(G,\F_p)$ that lifts to 
	$H^1(G,\Omega/I^2)$ lifts further to $H^1(G,\Omega/I^3)$.  Equivalently, this condition can be formulated as saying that if 
	$\chi \cup \lambda = \lambda \cup \psi = 0$, then $\partial(\lambda)=0$. 
	One can show that this is equivalent to showing that there exists $\nu \in H^1(G,\F_p)$ satisfying \eqref{eq:psi} that lies in
	the subgroup
	$$
		\ker(\chi\, \cup)+\ker(\psi \,\cup) + \ker((\chi+\psi)\,\cup).
	$$  
\end{remark}

\appendix
\section{Two lemmas from homological algebra}
\label{appendix}

We provide a proof of the following simple lemma from homological algebra for the reader's convenience.

\begin{lemma} \label{commsquare}
	Let $\mc{Q}$, $\mc{R}$,
	and $\mc{S}$ be abelian categories such that $\mc{Q}$ and $\mc{R}$ have enough projectives,
	and let $F \colon \mc{R} \to \mc{S}$ and
	$F' \colon \mc{Q} \to \mc{R}$ be right exact functors such that $F'$ sends projective objects to $F$-acyclic objects.  
	Let 
	$$
		0 \to A \to B \to C \to 0
	$$  
	be an exact sequence in $\mc{Q}$ such that
	$$
		0 \to F'(A) \to F'(B) \to F'(C) \to 0
	$$
	is exact.  For each $j \ge 0$, we have commutative diagrams
	$$
		\begin{tikzcd}
		L_{j+1}(F \circ F')(C) \arrow{r} \arrow{d} & L_j(F \circ F')(A) 
		\arrow{d} \\
		L_{j+1}F(F'(C)) \arrow{r} & L_jF(F'(A)),
		\end{tikzcd}
	$$
	in which the vertical
	arrows are edge maps in the Grothendieck spectral sequence attached to the composition $F \circ F'$
	and the horizontal maps are connecting morphisms,
	where $L_i$ denotes the $i$th left derived functor.
\end{lemma}

\begin{proof}
	Let $X$ denote any of $A$, $B$, and $C$.
	We may choose projection resolutions $P^X_{\cdot}$ 
	of $X$ with each term of
	$$
		0 \to F'(P^A_\cdot) \to F'(P^B_\cdot) \to F'(P^C_\cdot) \to 0
	$$ 
	split exact.  Then we may choose first quadrant 
	Cartan-Eilenberg resolutions
	$Q^X_{\cdot,\cdot}$ of the $F'(P^X_{\cdot})$ fitting in split exact sequences
	$$
		0 \to Q^A_{j,k} \to Q^B_{j,k} \to Q^C_{j,k} \to 0
	$$
	so that, in particular, we have exact sequences
	$$
		0 \to H_k(Q^A_{j,\cdot}) \to H_k(Q^B_{j,\cdot})
		\to H_k(Q^C_{j,\cdot}) \to 0,
	$$
	and the complexes
	$H_k(Q^X_{\cdot,\cdot}) \to H_k(F'(P_{\cdot}^X))$
	are projective resolutions.
	Note that 
	$$
		H_j(F(H_k(Q_{\cdot,\cdot}^X))) = L_jF(L_kF'(X)),
	$$ 
	and we have canonical isomorphisms
	$$
		H_j(F(\Tot Q_{\cdot,\cdot}^X)) \xrightarrow{\sim} H_j(F \circ F'(P_\cdot^X)) = L_j(F \circ F')(X),
	$$
	the first isomorphism as the terms of $F'(P_{\cdot}^X)$ are $F$-acyclic.
	The diagram in question is then simply
	$$
		\begin{tikzcd}
		H_{j+1}(F(\Tot Q_{\cdot,\cdot}^C)) \arrow{r} \arrow{d} & H_j(F(\Tot Q_{\cdot,\cdot}^A))
		\arrow{d} \\
		H_{j+1}(F(H_0(Q_{\cdot,\cdot}^C))) \arrow{r} & H_j(F(H_0(Q_{\cdot,\cdot}^A))),
		\end{tikzcd}
	$$
	the horizontal arrows being the connecting homomorphisms and the vertical arising from
	the augmentation maps on the total complexes.
\end{proof}

The following lemma is rather elementary but also useful to us.

\begin{lemma} \label{dontneedderived}
	Let $\mc{R}$ and $\mc{S}$ be abelian categories such that $\mc{R}$ has enough projectives.
	Let $F \colon \mc{R} \to \mc{S}$ be a left exact functor.  Suppose
	that $G \colon \mc{R} \to \mc{S}$ is a functor such that the pair $(F,G)$ extends to a functor from short exact sequences 
	$0 \to A \to B \to C \to 0$ in $\mc{R}$ to exact sequences
	$$
		G(A) \to G(B) \to G(C) \xrightarrow{\delta} F(A) \to F(B) \to F(C) \to 0.
	$$
	Then there is a natural transformation $G \rightsquigarrow L_1F$ such that the resulting diagrams
	$$
		\begin{tikzcd} 
		G(C) \arrow{dr}{\delta} \arrow{d} \\
		L_1F(C) \arrow[r,"\partial" near start] & F(A),
		\end{tikzcd}
	$$
	are commutative for the usual connecting homomorphisms $\partial$ and 	
	such that $G(A) \to L_1F(A)$ is an epimorphism for all objects $A$ of $\mc{R}$.
\end{lemma}

\begin{proof}
	Put any object $A$ of $\mc{R}$ in an exact sequence
	$$
		0 \to K \to P \to A \to 0
	$$
	in $\mc{R}$, where $P$ is a projective object.  We then have a commutative diagram
	$$
		\begin{tikzcd}
		G(P) \arrow{r} & G(A) \arrow{r} \arrow{d} & F(K) \arrow[equal, d] \arrow{r} & F(P) \arrow[equal,d]\\
		0 \arrow{r} & L_1F(A) \arrow{r} & F(K) \arrow{r} & F(P),
		\end{tikzcd}
	$$
	with exact rows, where the vertical morphism is unique making the diagram commute.
	That this gives a natural transformation is standard, and the fact that the morphisms are epimorphisms
	follows from the four lemmas.
\end{proof}

\begin{ack}
	This paper arose out of a project of the first, second, and fifth authors at the 2018 Arizona Winter School on Iwasawa theory 
	that was proposed by the third and led by the third and fourth authors.  We would like to thank the AWS for the stimulating
	environment that enabled our collaboration.  We would also like to thank Nguy\~{\^e}n Duy T\^an for helpful comments on 
	an earlier draft of this article, as well as the anonymous referee for a number of suggestions that helped us 
	to improve the exposition.
\end{ack}

\begin{fin}
	The second author's research was partially supported by the National Science Foundation under Grant No.\ DMS-2200541.
	The third author's research was supported in part by
	the National Science Foundation under Grant No.\  DMS-2101889.   The fourth author's research was supported in part by
	the National Science Foundation under Grant No.\  DMS-1901867.  The fifth author was partially supported by a Foerster-Berstein 
	Fellowship at Duke University and the National Science Foundation under Grant No.\ DMS-2201346.
\end{fin}

\renewcommand{\baselinestretch}{1}

\end{document}